\documentclass[12pt, a4paper]{amsart}
\usepackage{amsmath}
\usepackage{geometry,amsthm,graphics,tabularx,amssymb,shapepar}
\usepackage[cmyk]{xcolor}
\usepackage{appendix}
\usepackage{amscd}
\usepackage[all]{xypic}
\usepackage{ulem}
\usepackage{bm}

\makeatletter
\newcommand*{\rom}[1]{\expandafter\@slowromancap\romannumeral #1@}
\makeatother

\newcommand{\BC}{{\mathbb {C}}}

\newcommand{\BN}{{\mathbb {N}}}

\newcommand{\BZ}{{\mathbb {Z}}}

\newcommand{\CC}{{\mathcal {C}}}

\newcommand{\CL}{{\mathcal {L}}}

\newcommand{\CP}{{\mathcal {P}}}

\newcommand{\CR}{{\mathcal {R}}}

\newcommand{\CU}{{\mathcal {U}}}

\newcommand{\rd}{{\mathrm {d}}}

\newcommand{\wt}{\widetilde}
\newcommand{\wh}{\widehat}

\newcommand{\g}{\mathfrak g}

\newcommand{\h}{\mathfrak h}

\renewcommand{\b}{\mathfrak b}

\newcommand{\n}{\mathfrak n}

\newcommand{\f}{\mathfrak f}

\renewcommand{\t}{\mathfrak t}

\newcommand{\fh}{\mathfrak h}
\newcommand{\fn}{\mathfrak n}
\newcommand{\fg}{\mathfrak g}
\newcommand{\te}{\textrm}
\newcommand{\U}{\mathcal{U}}
\newcommand{\ot}{\otimes}

\newcommand{\Z}{\mathbb{Z}}
\newcommand{\C}{\mathbb{C}}

\newcommand{\cf}{\textit{cf}.~}

\newcommand{\be}{\begin {equation}}
\newcommand{\ee}{\end {equation}}
\newcommand{\bee}{\begin {equation*}}
\newcommand{\eee}{\end {equation*}}

\theoremstyle{Theorem}

\theoremstyle{Theorem}

\newtheorem{thm}{Theorem}[section]
\newtheorem{cort}[thm]{Corollary}
\newtheorem{lemt}[thm]{Lemma}
\newtheorem{prpt}[thm]{Proposition}
\newtheorem{thmt}[thm]{Theorem}
\newtheorem{remt}[thm]{Remark}

\theoremstyle{Theorem}

\theoremstyle{Theorem}

\theoremstyle{Plain}

\theoremstyle{remark}

\theoremstyle{remark}

\theoremstyle{Definition}

\newtheorem{dfnt}[thm]{Definition}

\numberwithin{equation}{section}

 \makeatletter

 \renewcommand*{\l@section}{%
   \addvspace{0.3em\@plus\p@}%
  \@tocline{1}{0pt}{1pc}{5pc}{}%
}

 \renewcommand*{\l@subsection}{%
   \@tocline{2}{0.2em}{2pc}{5pc}{}
}
 \makeatother
\begin{document}

\title[Parabolic induction for Takiff algebras]
{Parabolic induction for Takiff algebras and character formulas in Chari's category $\widetilde{\mathcal{O}}$}
\author{Fulin Chen}
\address{School of Mathematical Sciences, Xiamen University,
 Xiamen, China 361005}
  \email{chenf@xmu.edu.cn}

\author{Xin Huang}
\address{School of Mathematics, Hefei University of Technology, Hefei, China 230009}
 \email{xinhuang@hfut.edu.cn}

\author{Ruling Zhou}
\address{School of Mathematical Sciences, Xiamen University,
 Xiamen, China 361005}
  \email{zhouruling@stu.xmu.edu.cn}


\begin{abstract}
Let $\fg$ be a finite-dimensional complex semisimple Lie algebra, and let $\CC_N\fg=\fg\otimes (\C[t]/t^{N+1}\C[t])$ denote the $N$-th Takiff algebra of $\fg$. We introduce a notion of parabolic subalgebras of $\CC_N\fg$ and classify them in terms of $(N+2)$-partitions of a simple root system of $\fg$. We then study the corresponding parabolically induced $\CC_N\fg$-modules and establish a necessary and sufficient criterion for their irreducibility, extending Jantzen's criterion for parabolically induced $\fg$-modules. As an application, we derive character formulas for almost all irreducible Harish-Chandra modules in Chari's category $\widetilde{\mathcal{O}}$ over the (centerless) affine Lie algebra associated with $\fg$. In particular, when $\fg=\mathfrak{sl}_3(\C)$, we obtain explicit character formulas for all irreducible Harish-Chandra modules in this category.
\end{abstract}
\maketitle

 \tableofcontents
\section{Introduction}
In the representation theory of affine Lie algebras, the standard category $\mathcal O$ \cite{K2} and Chari’s category $\widetilde{\mathcal O}$ \cite{C} provide two natural analogues of the BGG category $\mathcal O$ for finite-dimensional semisimple Lie algebras. The irreducible objects of the former are standard highest weight modules, whereas those of the latter are highest weight modules relative to the natural Borel subalgebra, often called $\ell$-highest weight modules \cite{CG}. Both families arise in the study of integrable and unitarizable affine Lie algebra modules \cite{C,CP,CP1,JK1,JK2}. More precisely, Chari proved that every irreducible integrable module with finite-dimensional weight spaces is either a standard  highest/lowest weight module, or an $\ell$-highest weight module \cite{C}. Moreover, in Jakobsen and Kac’s classification of unitarizable highest weight modules for affine Lie algebras without the degree derivation, the nonstandard families consist of highest weight modules relative to natural Borel subalgebras \cite{JK1,JK2}.

Compared with standard highest weight modules, $\ell$-highest weight modules have a more complicated structure and are less well understood. Their weight spaces may be finite- or infinite-dimensional, whereas modules in the standard category $\mathcal O$ have finite-dimensional weight spaces. Weight modules with finite-dimensional weight spaces are commonly called Harish-Chandra modules. The classification of irreducible Harish-Chandra modules at nonzero level is well established \cite{CFM,F,F1,F2,F3,FK,FT}, while the zero level  case is more intricate. An important class at zero level  consists of the irreducible Harish-Chandra modules in $\widetilde{\mathcal O}$, whose classification is known \cite{BZ,CP}. 
Explicit character formulas for these $\ell$-highest weight modules are known in few cases beyond type $A_1^{(1)}$, for which a formula was obtained in \cite{W}. By comparison, the characters of irreducible standard highest weight modules  are given by the Weyl-Kac character formula in the integrable case \cite{K1} and described by Kazhdan-Lusztig theory at noncritical level \cite{KT}.
Determining character formulas for irreducible $\ell$-highest weight modules is therefore a fundamental problem.

The main goal of this paper is to derive character formulas for irreducible Harish-Chandra modules in the category $\widetilde{\mathcal O}$ for affine Lie algebras. Chari and Pressley's loop module construction \cite{CP} reduces this problem to computing the characters of irreducible highest weight modules with finite-dimensional weight spaces over loop algebras. As observed in \cite{CGT,R,W}, the latter problem reduces further to determining the characters of irreducible highest weight modules over Takiff algebras. Takiff algebras were introduced by Takiff \cite{T}, and their representation theory has since been extensively studied \cite{W2,Ch,ChT,MaS,He}.  The category $\mathcal O$, whose simple objects are precisely the irreducible highest weight modules, plays a central role in this theory \cite{Ch,ChT,MaS}. Thus, determining the characters of these modules is also a pressing problem in its own right within the representation theory of Takiff algebras.

Let $\mathfrak g$ be a finite-dimensional complex semisimple Lie algebra, and let $N$ be a nonnegative integer. Recall that the $N$-th Takiff algebra of $\mathfrak g$ is the truncated current Lie algebra
\[
\mathcal C_N\mathfrak g
:=\mathfrak g\otimes \bigl(\mathbb C[t]/t^{N+1}\C[t]\bigr).
\]
For $N=0$, Jantzen's irreducibility criterion \cite{J} is an important tool in the character theory of highest weight modules over $\CC_N\fg=\fg$. It determines when modules induced from one-dimensional representations of parabolic subalgebras are irreducible. Since these induced modules have explicit characters given by the Poincaré-Birkhoff-Witt theorem, the criterion yields character formulas for a broad class of irreducible highest weight $\mathfrak g$-modules. In fact, these formulas apply to almost all irreducible highest weight $\mathfrak g$-modules.
This motivates us to construct a suitable family of parabolically induced $\mathcal C_N\mathfrak g$-modules, and to establish an irreducibility criterion for these induced modules.

 In Section~\ref{sec:3}, we introduce a notion of parabolic subalgebra for $\mathcal C_N\mathfrak g$ (see Definition~\ref{def:para}). Such subalgebras need not be of the form
\[
\mathfrak p\otimes\bigl(\mathbb C[t]/t^{N+1}\C[t]\bigr),
\]
where $\mathfrak p$ is a parabolic subalgebra of $\mathfrak g$. Fix a system $\Pi$ of simple roots of $\mathfrak g$. When $N=0$, the parabolic subalgebras of $\mathcal C_N\mathfrak g=\mathfrak g$ are classified, up to conjugacy, by subsets of $\Pi$  \cite{Ja}. Each subset $I\subseteq\Pi$ determines the decomposition $\Pi=I\sqcup(\Pi\setminus I)$, which we regard as a $2$-partition of $\Pi$. For general $N$, we associate an $(N+2)$-partition of $\Pi$ with each parabolic subalgebra of $\mathcal C_N\mathfrak g$ (see Definition~\ref{def:N+2}). Although a fixed $(N+2)$-partition need not determine a unique parabolic subalgebra, it determines a minimal and a maximal member of the corresponding family. This leads to a classification of parabolic subalgebras of $\mathcal C_N\mathfrak g$ (see Proposition~\ref{clasubset}).

In Section~\ref{sec:4}, we introduce a class of parabolically induced $\mathcal C_N\mathfrak g$-modules. A natural approach to determining their irreducibility is to derive determinant formulas for their Shapovalov forms. This method has proved effective for Verma modules over Takiff algebras \cite{W2} and parabolically induced modules over semisimple Lie algebras \cite{J}. In the Takiff setting, however, the more intricate structure of parabolic subalgebras makes it difficult to obtain analogous determinant formulas by directly adapting the method of \cite{W2}.

Motivated by Jantzen's work \cite{J}, we first establish a sufficient condition for the irreducibility of parabolically induced $\mathcal C_N\mathfrak g$-modules (see Theorem~\ref{th:mainirr}). This condition also restricts the possible irreducible factors of the determinants of the associated Shapovalov forms. In Subsection~\ref{sec:5.3}, we analyze the leading terms of these determinants and derive a determinant formula under suitable hypotheses (see Theorem \ref{th:det}). This formula suffices to establish an irreducibility criterion for parabolically induced $\mathcal C_N\mathfrak g$-modules (see Theorem~\ref{thm:main}). 

Our main results on $\mathcal C_N\mathfrak g$ and its representations are as follows:
\begin{itemize}
\item a classification of its parabolic subalgebras
(Proposition~\ref{clasubset});
\item a necessary and sufficient condition for the irreducibility
of the parabolically induced modules
(Theorem~\ref{thm:main}); and
\item character formulas for almost all irreducible highest weight
modules in the category $\mathcal O$
(Theorem~\ref{th:chartak}).
\end{itemize}


 
In Section~\ref{sec:6}, we study the structure of irreducible highest weight $\mathcal L(\mathfrak g)$-modules with finite-dimensional weight spaces, where $\mathcal L(\mathfrak g):=\mathfrak g\otimes\mathbb C[t,t^{-1}]$ denotes the loop algebra associated with $\mathfrak g$. Such modules are determined by certain exp-polynomial characters and can be realized as tensor products of irreducible highest weight modules over suitable Takiff algebras. Using this decomposition, we derive character formulas for a class of irreducible Harish-Chandra $\mathcal L(\mathfrak g)$-modules  (see Theorem \ref{th:char-loop}). 

In Section~\ref{sec:7}, we combine the preceding results with Chari and Pressley's loop module construction \cite{CP} to obtain character formulas for a large class of irreducible Harish-Chandra modules in Chari's category $\widetilde{\mathcal O}$ (see Theorem \ref{th:cha}). As in the $\mathfrak{sl}_2(\mathbb C)$ case \cite{W}, the resulting formulas are expressed in terms of the characters of irreducible highest weight $\mathcal L(\mathfrak g)$-modules, with coefficients involving Euler's totient function and the Möbius function.

Finally, in Section~\ref{sec:8}, we specialize to $\fg=\mathfrak{sl}_3(\C)$ and determine the maximal proper submodules of reducible parabolically induced $\CC_N\fg$-modules. Combining this description with Theorem~\ref{th:chartak}, we obtain character formulas for all irreducible highest weight $\CC_N\fg$-modules and, consequently, for all irreducible Harish-Chandra modules in $\widetilde{\mathcal O}$ for the corresponding affine Lie algebra.
The analogous character formulas for $\fg=\mathfrak{sl}_2(\C)$ were obtained in \cite{W,W2}. In this rank-one case, every irreducible highest weight $\CC_N\fg$-module is either the pullback of an irreducible Verma module over $\CC_m\fg$ for some $0\leq m\leq N$,
or the pullback of a finite-dimensional irreducible
$\fg$-module (cf.~\cite{W2}). In higher rank, determining the characters of irreducible highest weight modules requires further analysis of the submodule structure of reducible parabolically induced modules.

Throughout this paper, let $\C$, $\C^\times$, $\Z$, $\mathbb{N}$ and $\Z_+$ denote the sets of complex numbers, nonzero complex
numbers, integers, non-negative integers and positive integers, respectively. Let $\mathcal{U}(\cdot)$ denote the universal enveloping algebra of a Lie algebra. For every 
$\Gamma$-graded Lie algebra 
$\mathfrak{l}=\oplus_{\gamma\in \Gamma} \mathfrak{l}_\gamma$, write
\[
\mathrm{supp}_\Gamma(\mathfrak{l}):=\{\gamma\in \Gamma\mid \mathfrak{l}_\gamma\ne 0\}\]
for the support of $\mathfrak{l}$ in $\Gamma$, where 
$\Gamma$ is an abelian (semi-)group.

\section{Preliminaries}\label{sec:2}
In this section, we introduce some basic notation on complex semisimple Lie algebras and Takiff algebras.
\subsection{Semisimple Lie algebras}

Throughout this paper, let $\mathfrak{g}$ be a finite-dimensional complex semisimple Lie algebra and let $\fh$ be a fixed Cartan subalgebra of $\fg$. Write $\Delta\subset \h^*$ and $Q=\text{Span}_\Z\Delta$ for  the root system and the root lattice of $\mathfrak{g}$ (relative to $\mathfrak{h}$), respectively.
Then \[
\fg=\bigoplus_{\alpha\in Q}\fg_\alpha \qquad  (\fg_\alpha:=\{x\in\fg\mid [h,x]=\alpha(h)x\ \text{for all }h\in\fh\})
\]
is a $Q$-grading of $\fg$, with  $\fg_0=\fh$ and $\mathrm{supp}_Q(\fg)=\Delta\cup\{0\}$.

Fix a base $\Pi$ of $\Delta$, and  set
\[
Q^\pm=\textrm{Span}_{\BN}\pm\Pi\quad\textrm{and}\quad \Delta^\pm=\Delta\cap Q^\pm.\]
Then $\fg$ admits a triangular decomposition
\[\fg=\fn^-\oplus\fh\oplus\fn^+,\]
where $\n^\pm:=\sum_{\alpha\in \Delta^\pm}\mathfrak{g}_\alpha$. We also set 
\[\b:=\fh\oplus\fn^+,\] which is a Borel subalgebra of $\fg$.

Let $(\,\cdot\,,\,\cdot\,)$ be the  Killing form of $\fg$. 
For each $\alpha\in\Delta$, write 
\[
\alpha^\vee:=\frac{2t_\alpha}{(t_\alpha,t_\alpha)}\in\mathfrak h
\]
for the coroot of $\alpha$, where
$t_\alpha\in\mathfrak h$ is uniquely determined by 
\[
\alpha(h)=(t_\alpha,h)\quad \text{for all } h\in\mathfrak h.
\]
Fix a Chevalley anti-involution $\omega$ of $\mathfrak{g}$ such that
$\omega|_{\mathfrak{h}}=\operatorname{id}_{\mathfrak{h}}$.
For each $\alpha\in\Delta^+$, choose $x_\alpha\in\mathfrak{g}_\alpha$ such that the triple $(x_\alpha,\alpha^\vee,x_{-\alpha})$ is an $\mathfrak{sl}_2$-triple, where $x_{-\alpha}:=\omega(x_\alpha)$.

\subsection{Takiff algebras}
Let $\mathcal{R}$ be a commutative associative unital $\mathbb C$-algebra. For any Lie subalgebra $\mathfrak{a}$ of $\g$, set
\[
\mathcal{R}\mathfrak{a}:=\mathfrak{a}\otimes \mathcal{R},
\]
endowed with the Lie bracket 
\[
[x\otimes r, y\otimes s]=[x,y]\otimes rs\quad \textrm{for } x,y\in \mathfrak{a} \textrm{ and } r,s\in \mathcal{R}.
\]
Throughout this paper, we identify $\mathfrak{a}$ with the subalgebra
$\mathfrak{a}\otimes 1$ of $\mathcal{R}\mathfrak{a}$ via the embedding $x\mapsto x\otimes 1$. 

Let $N$ be a nonnegative integer, and set 
\[
\CC_N:=\C[t]/t^{N+1}\C[t].
\] 
The Lie algebra $\CC_N\mathfrak{g}$ is  called the $N$-th Takiff algebra of $\mathfrak{g}$. The natural quotient map $\C[t]\to \CC_N$ induces a surjective Lie algebra homomorphism 
\begin{equation}\label{eq:quoT}
\C[t]\mathfrak g\longrightarrow \CC_N\mathfrak g.
\end{equation}
For every $x\in\mathfrak{g}$ and $k\in \BN$, write $x(k)$ for the image of $x\otimes t^k$ under the  map \eqref{eq:quoT}. 
Note that $x(k)=0$ for $k>N$.

The natural $\BN$-grading on $\C[t]$ by degree in $t$  induces an  $\BN$-grading
on $\CC_N$ and hence   on $\CC_N\mathfrak{g}$:
\[
\CC_N\mathfrak{g}=\bigoplus_{n\in\BN} (\CC_N\mathfrak{g})_{n},
\]
where $(\CC_N\mathfrak{g})_{n}:=\{x(n)\mid x\in \mathfrak{g}\}$.
On the other hand, the adjoint action of $\fh$ on $\CC_N\fg$ induces a $Q$-grading on $\CC_N\fg$ such that 
\begin{align*}
    \CC_N\fg=\bigoplus_{\alpha\in Q}(\CC_N\fg)_\alpha,
\end{align*}
where $ (\CC_N\mathfrak{g})_\alpha := \g_{\alpha }\otimes \CC_N$.
These two gradings are compatible and hence give rise to a  $(Q\times \BN)$-grading
\[\CC_N\mathfrak{g}=\bigoplus_{(\alpha,n)\in Q\times \BN} (\CC_N\mathfrak{g})_{(\alpha,n)},\]
where $(\CC_N\mathfrak{g})_{(\alpha,n)}:=(\CC_N\fg)_\alpha\cap (\CC_N\fg)_{n}$. Explicitly, we have 
\begin{align*}
(\CC_N\mathfrak{g})_{(\alpha,n)}=
    \begin{cases}
     \{h(n)\mid h\in\fh\}, &\mbox{ if } \alpha=0,\ 0\le n\le N,\\
     \C x_{\alpha}(n), &\mbox{ if }\alpha\in\Delta,\ 0\le n\le N,\\
     0, &\mbox{ otherwise.}
    \end{cases}
\end{align*}
The support of this $(Q\times\BN)$-grading  is
\[
\widetilde{\Delta}_N:=\mathrm{supp}_{Q\times \BN}(\CC_N\mathfrak{g})=
(\Delta\cup \{0\})\times \{0,1,\dots,N\}.\]

\subsection{Highest weight $\CR\fg$-modules}
Let $V$ be a $\CR\mathfrak{g}$-module.
It is called a weight module  if it admits a weight space decomposition with respect to $\fh$:
\[
V=\bigoplus_{\lambda\in \h^*} V_\lambda\quad (V_\lambda:=\{v\in V\mid h.v=\lambda(h)v ~\text{for all}~h\in \mathfrak{h}\}).
\]
A nonzero vector $v\in V$ is called a weight vector of weight $\lambda$ if $v\in V_\lambda$, in which case we write $\textrm{wt\,}v=\lambda$.
If $V$ is a weight module, then it is called a Harish-Chandra module if  $\dim V_\lambda<\infty$ for all $\lambda\in  \h^*$.
In this case, as usual, the character of  $V$ is defined by
\begin{equation*}
\textrm{char}\,V=\sum_{\lambda\in\fh^{\ast}}\textrm{dim\,}V_{\lambda}\cdot e(\lambda),
\end{equation*}
where $e(\lambda)$ denotes the formal exponential corresponding to $\lambda$.
\begin{dfnt}
    Let $\chi\in (\CR\fh)^{\ast}$. We say that a $\CR\fg$-module $V$ is a highest weight module with highest weight $\chi$ if there exists a nonzero vector $v\in V$, called a highest weight vector, such that 
    \[
    V=\U(\CR\fg).v,\quad \CR\fn^+.v=0,\quad\text{and}\quad h.v=\chi(h)v,\ \forall\ h\in\CR\fh.
    \]
\end{dfnt}
Let $\CR'$ be another commutative associative unital $\C$-algebra and let $\eta:\CR\rightarrow \CR'$ be a unital $
\C$-algebra homomorphism. 
 Then $\eta$ induces a Lie  homomorphism 
\[
\ \CR\fg\rightarrow \CR'\fg,\quad x\otimes r\mapsto x\otimes \eta(r),
\] 
which, by abuse of notation, will also be denoted by $\eta$.
For every $\CR'\fg$-module $W$, the pullback along $\eta$ induces a $\CR\fg$-module structure on $W$ such that 
\[(x\otimes r).w= (x\otimes \eta(r)).w\quad \textrm{for }x\in\fg,~r\in\CR,~w\in W.\]
Denote by $\eta^{-1}(W)$ the resulting $\CR\fg$-module.
Similarly, for $\chi'\in (\mathcal{R}'\fh)^*$, define $\eta^{-1}(\chi')$ by
  \[\left(\eta^{-1}(\chi')\right)(h\otimes r)=\chi'(h\otimes\eta(r))\quad \textrm{for }h\in\fh,~r\in \mathcal{R}.\]

For later use, we record the following straightforward observation. 
  \begin{lemt}\label{pullback}
  Let $W$ be a highest weight $\CR'\fg$-module of highest weight $\chi'$. If $\eta$ is surjective, then $\eta^{-1}(W)$ is a highest weight $\CR\fg$-module with highest weight $\eta^{-1}(\chi')$. Moreover, if $W$ is irreducible, then $\eta^{-1}(W)$ is irreducible.
  \end{lemt}


\section{Parabolic subalgebras of $\CC_N\fg$}\label{sec:3}
In this section, we introduce a notion of parabolic subalgebra of $\CC_N\fg$ and 
give a classification of these subalgebras.
\subsection{Parabolic subalgebras, parabolic subsets, and partitions}
Recall that a Lie subalgebra of $\mathfrak{g}$ is called parabolic if it contains a  Borel subalgebra  of $\g$.
Motivated by this notion, we introduce the following definition. 

\begin{dfnt}\label{def:para}
A Lie subalgebra $\mathcal{P}$ of $\CC_N\fg$ is called parabolic if it is $\BN$-graded and there exists a Borel subalgebra $\b'$ of $\g$ such that $\CC_N\mathfrak b' \subseteq \mathcal{P}$.
\end{dfnt}

Since all Borel subalgebras of $\fg$ are conjugate, every parabolic subalgebra of $\CC_N\fg$ is conjugate to one containing $\CC_N \b$. Hence, throughout the rest of the paper, unless otherwise stated, we restrict our attention to parabolic subalgebras containing $\CC_N\b$.

Let $\mathcal P$ be a parabolic subalgebra containing
$\CC_N\b$. 
Then  $\mathcal{P}$ is naturally 
a $(Q\times \BN)$-graded subalgebra of $\CC_N\mathfrak{g}$.
Set \[
\widetilde\Delta_{\mathcal{P}}:=\mathrm{supp}_{Q\times \BN}(\mathcal{P}).\]
It follows immediately that $\widetilde\Delta_{\mathcal{P}}$ is a parabolic subset of $\widetilde{\Delta}_N$ in the following sense: 
\begin{dfnt}\label{defparaset} A subset $S$ of $\widetilde{\Delta}_N$ is called parabolic if 
\begin{itemize}
    \item $ \widetilde{\Delta}_N^+:=
(\Delta^+\cup \{0\})\times \{0,1,\dots,N\}
 \subseteq S$, and 
    \item $S$ is closed in the sense that whenever
$(\alpha,i),(\beta,j)\in S$
and
$(\alpha,i)+(\beta,j)\in\widetilde{\Delta}_N$,
one has
$(\alpha,i)+(\beta,j)\in S$.
\end{itemize}    
\end{dfnt}
Conversely, let $S$ be a parabolic subset of $\widetilde{\Delta}_N$. Set
\begin{align*}
    \mathcal{P}_S:=\bigoplus_{(\alpha,n)\in S} (\CC_N\fg)_{(\alpha,n)}.
\end{align*}
Since $\widetilde{\Delta}_N^+\subseteq S$ and $S$ is closed, it follows that  $\mathcal{P}_S$ is a subalgebra of $\CC_N\g$ containing $\CC_N\b$, and $ \mathcal{P}_S$ is a parabolic subalgebra of $\CC_N\g$ with $\widetilde{\Delta}_{\mathcal{P}_S}=S$. Thus, the assignments above establish a one-to-one correspondence between parabolic subalgebras of $\CC_N\fg$ containing $\CC_N\b$ and parabolic subsets of $\widetilde{\Delta}_N$.
Therefore, the classification of parabolic subalgebras of $\CC_N\fg$ reduces to that of parabolic subsets of $\widetilde{\Delta}_N$.

When $N=0$, the parabolic subsets of $\widetilde{\Delta}_0$ are in one-to-one correspondence with the subsets of $\Pi$ (see \cite{Ja}), or equivalently, with the set of $2$-partitions of $\Pi$ in the sense of Definition \ref{def:N+2} below.

\begin{dfnt}\label{def:N+2}
An $(N+2)$-partition of a set $T$ is an ordered $(N+2)$-tuple
\[
(T_0,T_1,\dots,T_{N+1}),
\]
where each $T_i$ is a (possibly empty) subset of $T$, such that
\[
T=\bigcup_{i=0}^{N+1}T_i\quad \textrm{and}\quad T_j\cap T_k=\emptyset
\qquad\text{for all } 0\leq j\neq k\leq N+1.
\]
\end{dfnt}

We will show that every parabolic subalgebra of  $\CC_N\fg$ has a type determined by an $(N+2)$-partition of $\Pi$. 
Let $S$ be a parabolic subset of $\widetilde{\Delta}_N$. In what follows, we first associate  an $(N+2)$-partition of $\Pi$ to $S$. For every $\alpha\in \Delta^+$, set 
\begin{align*}
     d_S(\alpha)=\begin{cases}
         \min\{0\le n\le N\mid (-\alpha,n)\in S\},&\mbox{if }(\{-\alpha\}\times \{0,1,\dots,N\})\cap S\ne \emptyset,\\
         N+1,&\mbox{otherwise.}
     \end{cases}
\end{align*}
This defines a map 
\[
d_S :\  \Delta^+ \to \{0,1,\dots,N+1\}, \quad \alpha \mapsto d_S(\alpha).
\] 
Note that the map $d_S$ uniquely determines $S$ as 
\begin{align}\label{eq:S}
     S= \widetilde{\Delta}_N^+ \cup 
\{(-\alpha,j)\mid \alpha\in \Delta^+, ~ d_S(\alpha)\le j\le N\}.
\end{align}
The restriction of $d_S$ to $\Pi$ determines an $(N+2)$-partition 
\[\bm\Pi_S=(\Pi_{S,0},\Pi_{S,1},\dots,\Pi_{S,N+1})\]
of $\Pi$, where
\[\Pi_{S,i}=\{\alpha\in \Pi\mid d_S(\alpha)=i\}\quad \textrm{for }i=0,1,\dots,N+1.\]

For a parabolic subalgebra $\mathcal{P}$ of $\CC_N\fg$, set
\[ \bm\Pi_{\mathcal{P}} := \bm\Pi_{\widetilde{\Delta}_{\mathcal{P}}}.  \]
\begin{dfnt}
Let $\bm\Pi=(\Pi_0,\Pi_1,\dots,\Pi_{N+1})$ be an $(N+2)$-partition of $\Pi$. 
   A parabolic subset $S$ of  $\widetilde{\Delta}_N$ is said to be  of type $\bm\Pi$ if  $\bm\Pi_S = \bm\Pi$, and 
   a parabolic subalgebra $\mathcal{P}$ of $\CC_N\fg$  is said to be of type $\bm\Pi$ if  $ \bm\Pi_{\mathcal{P}} = \bm\Pi$.
\end{dfnt}

\subsection{Structure of parabolic subalgebras}
Fix an $(N+2)$-partition 
\[\bm\Pi=(\Pi_0,\Pi_1,\dots,\Pi_{N+1})\] of $\Pi$.
There is a unique homomorphism of abelian groups
\[
d_{\bm\Pi}:\ Q\rightarrow \Z,\quad \alpha\mapsto d_{\bm\Pi}(\alpha)
\]
such that
\[
d_{\bm\Pi}(\alpha)=i
\quad
\text{for every }\alpha\in\Pi_i.
\]
Set 
\begin{align*}
\mathcal{P}^{\bm\Pi}:=\CC_N\mathfrak{b}\oplus\sum_{\alpha\in \Delta^+}\sum_{j=d_{\bm\Pi}(\alpha)}^{N}\C x_{-\alpha}(j).
\end{align*}
Here and below, a sum over an empty index set is understood to be zero. Since $d_{\bm\Pi}$ is additive, one easily verifies that $\mathcal{P}^{\bm\Pi}$ is a parabolic subalgebra of $\CC_N\mathfrak{g}$ of type $\bm\Pi$ with 
\[\widetilde{\Delta}_{\mathcal{P}^{\bm\Pi}}= \widetilde{\Delta}_N^+ \cup  \{(-\alpha,j)\mid \alpha\in \Delta^+,~d_{\bm\Pi}(\alpha)\le j\le N\}. \]

For $0 \le i \le N+1$, let $\Delta_i$ and $\bar{\Delta}_i$ denote the root subsystems of $\Delta$ generated by $\Pi_i$ and $\overline{\Pi}_i:=\cup_{j=0}^i\Pi_j$, respectively. In particular, 
\[\Delta_0=\bar{\Delta}_0\quad\te{and}\quad \bar{\Delta}_{N+1}=\Delta.\]
Note that 
\begin{align*}
    \bar{\Delta}_0\subseteq \bar{\Delta}_1\subseteq\cdots\subseteq \bar\Delta_{N+1},
\end{align*}
hence\[
\hat\Delta_0 := \bar\Delta_0,\qquad
\hat\Delta_i:=\bar\Delta_i\setminus\bar\Delta_{i-1}
\quad (1 \le i\le N+1)
\] 
define an $(N+2)$-partition
\[(\hat\Delta_0,\hat\Delta_1,\dots,\hat\Delta_{N+1})\]
of $\Delta$.
Observe that 
 $\Delta_i\subseteq \hat\Delta_i$. In addition, we set 
\[
\Delta^\pm_i:=\Delta^\pm\cap \Delta_i,\quad \bar\Delta_i^\pm:=\Delta^\pm\cap \bar\Delta_i,\quad\text{and}\quad 
\hat\Delta_i^\pm:=\Delta^\pm\cap \hat\Delta_i.
\] 

Define a map 
 \[
D_{\bm\Pi}: \Delta^{+} \to \{0,1,\ldots,N+1\}\]
by letting 
\[D_{\bm\Pi}(\alpha)=i\quad \textrm{whenever }\alpha\in \hat\Delta_i^+.\]
It is straightforward to see that 
\[\mathcal{P}_{\bm\Pi}=\CC_N\mathfrak{b}\oplus\sum_{\alpha\in \Delta^+}\sum_{j=D_{\bm\Pi}(\alpha)}^{N}\C x_{-\alpha}(j)\]
is a parabolic subalgebra of $\CC_N\fg$ of type $\bm\Pi$ with
\[\widetilde{\Delta}_{\mathcal{P}_{\bm\Pi}}= \widetilde{\Delta}_N^+ \cup  \{(-\alpha,j)\mid \alpha\in \Delta^+,~D_{\bm\Pi}(\alpha)\le j\le N\}.\]

\begin{prpt}\label{clasubset}
Let $\bm\Pi$ be an $(N+2)$-partition of $\Pi$ and let $S$ be a parabolic subset of $\widetilde{\Delta}_N$ of type $\bm\Pi$. 
Then $\widetilde{\Delta}_{\mathcal{P}^{\bm\Pi}}\subseteq S\subseteq \widetilde{\Delta}_{\mathcal{P}_{\bm\Pi}}$. 
In particular, if $\mathcal{P}$ is a parabolic subalgebra of $\CC_N\fg$ of type $\bm\Pi$, then 
    \begin{align*}
        \mathcal{P}^{\bm\Pi}\subseteq \mathcal{P}\subseteq\mathcal{P}_{\bm\Pi}.
    \end{align*}
\end{prpt}
\begin{proof}
Let $\alpha\in\Delta^{+}$. By definition, it suffices to prove the inequalities
\[
d_S(\alpha)\leq d_{\bm\Pi}(\alpha) \quad\te{if}\ d_{\bm\Pi}(\alpha) \le N,\]
and 
\[
D_{\bm\Pi}(\alpha) \leq d_S(\alpha) \quad\te{if}\ d_S(\alpha) \le N.\]

Choose simple roots $
\alpha_1,\alpha_2,\dots,\alpha_k\in\Pi$
such that
\[
\alpha=\alpha_1+\cdots+\alpha_k
\]
and every partial sum
\[
\alpha_1+\alpha_2+\cdots+\alpha_t\in\Delta^+
\quad\textrm{for }
1\le t\le k.
\]
Then 
\begin{align}\label{eq:dpialpha}
  d_{\bm\Pi}(\alpha)=d_{\bm\Pi}(\alpha_1)+d_{\bm\Pi}(\alpha_2)+\dots+d_{\bm\Pi}(\alpha_k).  
\end{align}
For the first inequality, since $S$ is of type $\bm\Pi$, we have  
\begin{align*}
    (-\alpha_j, d_{\bm\Pi}(\alpha_j))=(-\alpha_j, d_{S}(\alpha_j))\in S \quad ( 1 \le j \le k).
\end{align*}
Using the closedness of $S$ successively along the chosen partial sums, we obtain
\begin{align*}
    (-\alpha,d_{\bm\Pi}(\alpha))=(-\alpha_1,d_{\bm\Pi}(\alpha_1))+(-\alpha_2,d_{\bm\Pi}(\alpha_2))+\dots+(-\alpha_k,d_{\bm\Pi}(\alpha_k))\in S.
\end{align*}
This gives $d_S(\alpha)\leq d_{\bm\Pi}(\alpha)$.

It remains to show that $D_{\bm\Pi}(\alpha) \leq   d_S(\alpha)$. We prove it by induction on $k$. The case that $k=1$ is straightforward. Assume that $k>1$.
Set
\begin{align*}
A=\{\alpha_j\mid 1\leq j\leq k,~\alpha-\alpha_j\in \Delta\}.
\end{align*}
Clearly, $A\neq \emptyset$. For every $\beta\in A$, one has \begin{align*}
    (-\beta, d_S(\alpha))=(-\alpha, d_S(\alpha))+(\alpha-\beta,0)\in S.
\end{align*}
It follows that $d_S(\beta)\leq d_S(\alpha)$. Then the assertion holds if 
there exists $\alpha_j\in A$ such that 
$d_S(\alpha_j)=D_{\bm\Pi}(\alpha)$. 

Otherwise, we have
\begin{align}\label{<}
    d_S(\beta)<D_{\bm\Pi} (\alpha)\quad \te{for every }\beta\in A
\end{align}
by noting that 
\begin{align}\label{eqmax}
  D_{\bm\Pi}(\alpha)=\textrm{max}\{D_{\bm\Pi}(\alpha_j)\mid j=1,2,\dots,k\}=\textrm{max}\{d_{S}(\alpha_j)\mid j=1,2,\dots,k\}.  
\end{align}
Fix a $\beta\in A$. Since $S$ is closed,  we obtain
\[(-\alpha+\beta,d_S(\alpha))=(-\alpha,d_S(\alpha))+(\beta,0)\in S.\]
Then $d_{S}(\alpha-\beta)\leq d_S(\alpha)$.
On the other hand, it follows from \eqref{<} and \eqref{eqmax} that 
\[D_{\bm\Pi}(\alpha)=D_{\bm\Pi}(\alpha-\beta).\]
The induction hypothesis yields that $D_{\bm\Pi}(\alpha-\beta)\leq d_S(\alpha-\beta)$, and then the assertion follows immediately.
\end{proof}
\begin{remt}\label{remdec}
For every parabolic subalgebra $\mathcal{P}$ of $\CC_N\fg$ of type $\bm\Pi$ with support $S=\widetilde\Delta_{\mathcal{P}}$, there is a direct sum decomposition of $\CC_N\fg$ into vector subspaces:
\begin{align}\label{decpara}
\CC_N\fg=\fn_{\bm\Pi,S}^-\oplus \fg_{\bm\Pi,S}\oplus \fn_{\bm\Pi,S}^+
\end{align}
with $\mathcal{P}=\fg_{\bm\Pi,S}\oplus \fn_{\bm\Pi,S}^+$, where
\[
\fg_{\bm\Pi,S}
=\CC_N\fh
\oplus
\sum_{\alpha\in\Delta^+}\sum_{j=d_S(\alpha)}^N
\bigl(\C x_\alpha(j)+\C x_{-\alpha}(j)\bigr),\quad \fn_{\bm\Pi,S}^{\pm}
=
\sum_{\alpha\in\Delta^+}\sum_{j=0}^{d_S(\alpha)-1}
\C x_{\pm\alpha}(j).
\]
When $\mathcal{P}=\mathcal{P}^{\bm\Pi}$, the additivity of $d_{\bm\Pi}$ implies that all three summands in \eqref{decpara} are subalgebras of $\CC_N\fg$. On the other hand, when $\mathcal{P}=\mathcal{P}_{\bm\Pi}$, we have $d_S=D_{\bm\Pi}$. In this case, $\fg_{\bm\Pi,S}$ is a subalgebra of $\CC_N\fg$, whereas $\fn_{\bm\Pi,S}^{\pm}$ are not subalgebras in general. Indeed,  if $\alpha\in\hat\Delta_m^\pm$ and $\beta\in\hat\Delta_n^\pm$ satisfy $\alpha+\beta\in\Delta$ for some $m,n>1$ and $m+n\leq N+2$, then
\[
[x_{\alpha}(m-1),x_{\beta}(n-1)]\notin \fn_{\bm\Pi,S}^{\pm}.
\]
\end{remt}


\section{Parabolically induced modules over Takiff algebras}\label{sec:4}
In this section, we fix an $(N+2)$-partition $\bm\Pi$ of $\Pi$ and a parabolic subalgebra $\CP$ of $\CC_N\mathfrak{g}$ of type $\bm\Pi$. We introduce parabolically induced $\CC_N\mathfrak{g}$-modules associated with $\CP$ and study their irreducibility.

\subsection{Induced modules}
  Let $\f$ be a complex Lie algebra and $\t$ a subalgebra of $\f$.
  For a $\t$-module $V$, define the induced $\f$-module by
  \[
\mathrm{Ind}_{\t}^{\f}\, V=\mathcal{U}(\f)\otimes_{\CU(\t)} V.
\]
In this paper, we only consider the case where $V$ is one-dimensional.

A character of $\mathfrak{t}$ is a linear functional $\chi$ on  $\mathfrak{t}$ such that $\chi([\mathfrak{t},\mathfrak{t}])=0$. For every  character $\chi$ of $\t$, there is a one-dimensional $\t$-module \[
\BC_\chi=\BC v_\chi
\] such that 
\[av_\chi=\chi(a) v_\chi\quad \te{ for all } a\in \t.\] 
Conversely,  every one-dimensional $\t$-module is of this form, up to isomorphism.

Consider the following direct sum decomposition of $\CC_N\fh$:
\begin{align}\label{hpi}
    \CC_N\fh=\mathfrak{h}_{\bm\Pi}\oplus \h_{\bm\Pi}^\perp.
\end{align}
 where \[
\mathfrak{h}_{\bm\Pi}:=\sum_{i=1}^{N+1}\sum_{\alpha\in \Pi_i}\sum_{j=0}^{i-1}  \BC \alpha^\vee(j)\quad\te{and}\quad \h_{\bm\Pi}^\perp : =  \sum_{i=0}^{N+1}\sum_{\alpha\in \Pi_i}\sum_{j=i}^{N}  \BC \alpha^\vee(j).\]

\begin{lemt}\label{le:pp}
Let $\CP$ be a parabolic subalgebra of $\CC_N\g$ of type $\bm\Pi$. Then there is a linear isomorphism 
    \[
\mathcal{P}/[\mathcal{P},\mathcal{P}]\cong \mathfrak{h}_{\bm\Pi}.
\]
\end{lemt}
\begin{proof}
Note that 
\[\CP=\CC_N\fh\oplus \bigoplus_{\substack{(\alpha,n)\in \widetilde\Delta_{P},\\ \alpha\neq 0}}\C x_{\alpha}(n).\]
Then we have 
\[x_{\alpha}(n)\in [\fh, x_{\alpha}(n)]\subseteq [\CP,\CP]\quad \textrm{for }(\alpha,n)\in \widetilde\Delta_{P}\textrm{ with }\alpha\neq 0,\]
 and  it follows that 
\begin{align}\label{pp}
\mathcal{P}/[\CP,\CP]\cong \CC_N\fh/\left([\mathcal{P},\mathcal{P}]\cap \CC_N\fh\right).    
\end{align}
Since $\bm\Pi_{\tilde\Delta_{P}}=\bm\Pi$, one easily verifies that 
\[[\mathcal{P},\mathcal{P}]\cap \CC_N\fh=\h_{\bm\Pi}^\perp.\]
Then the assertion follows from \eqref{hpi} and \eqref{pp}.
\end{proof}

A character on $\CP$ is equivalent to  a linear functional on $\CP/[\CP,\CP]$. In view of Lemma \ref{le:pp}, we shall therefore identify characters of parabolic subalgebras of type $\bm\Pi$ with elements of
$\mathfrak{h}_{\bm\Pi}^*
$
throughout the rest of the paper.

\subsection{Parabolically induced $\CC_N\g$-modules}\label{sec:paramod}
Let $\chi$ be a  character of $\CP$. 
Form the induced $\CC_N\fg$-module 
\[
M_{\CP}(\chi):=\mathrm{Ind}_{\CP}^{\CC_N\g}~\C_\chi.
\]
When $\mathcal{P}=\CC_N\b$ (and hence $\fh_{\bm\Pi}=\CC_N\fh$), it is the usual Verma $\CC_N\fg$-module with highest weight  $\chi$ (\cf\,\cite{W2}).
For every $\lambda\in\fh^*$, denote by
\[
M_{\CP}(\chi)_\lambda
:=
\{v\in M_{\CP}(\chi)\mid h.v=\lambda(h)v
\text{ for all }h\in\fh\}
\]
 the weight space of weight $\lambda$. Then $M_{\CP}(\chi)$ is a $\h$-weight module with weight space decomposition
\[
M_{\CP}(\chi) = 
\bigoplus_{\alpha\in Q^+}
M_{\CP}(\chi)_{\chi|_{\fh}-\alpha},
\]
and $M_{\CP}(\chi)_{\chi|_{\fh}}=\C v_\chi$. Consequently, $M_{\CP}(\chi)$ has a unique irreducible quotient, which we denote by $V_{\CP}(\chi)$.
Since $\CP^{\bm\Pi}\subseteq \CP\subseteq\CP_{\bm\Pi}$, we have surjective homomorphisms
\[
M_{\CP^{\bm\Pi}}(\chi)\rightarrow M_\CP(\chi)\quad \text{and}\quad M_\CP(\chi)\rightarrow M_{\CP_{\bm\Pi}}(\chi),
\]
as well as isomorphisms
\[
V_{\CP_{\bm\Pi}}(\chi)\cong V_\CP(\chi)\cong V_{\CP^{\bm\Pi}}(\chi).
\]
In particular, the parabolically induced $\CC_N\fg$-module is irreducible only when $\CP=\CP_{\bm\Pi}$. 
Thus, we just consider the case that  $\CP=\CP_{\bm\Pi}$ in the rest of this paper. 

For later use, we need the following notion.
\begin{dfnt}
Let $\bm\Pi=(\Pi_0,\Pi_1,\dots,\Pi_{N+1})$ be an $(N+2)$-partition of $\Pi$. We say that
 $\bm\Pi$ is trivial if $\Pi_0=\Pi$, and  $\bm\Pi$ is genuine if $\Pi_{N+1}\neq \emptyset$.
\end{dfnt}
If $\bm\Pi$ is trivial, then
$\CP_{\bm\Pi}=\CC_N\fg$ and $\fh_{\bm\Pi}=0$. Hence the only character is the zero character, and
$M_{\CP_{\bm\Pi}}(0)$ is the one-dimensional trivial
$\CC_N\fg$-module. When $\bm\Pi$ is nontrivial, it suffices to treat the case that $\bm\Pi$ is genuine. Indeed, 
assume that $\bm\Pi$ is not genuine. Denote by $M$ the largest  integer such that 
$\Pi_{M+1}\ne \emptyset$. Then 
\[
\bm\Pi'=(\Pi_0,\Pi_1,\dots,\Pi_{M+1})
\]
is a genuine $(M+2)$-partition of $\Pi$. Note that the natural epimorphism of algebras 
\begin{align}\label{piMN}
    \pi_{N,M}:\ \CC_N\rightarrow \CC_M
\end{align}
induces a natural surjective Lie algebra homomorphism 
\begin{eqnarray*}
    \tilde{\pi}_{N,M}:\ \CC_N\mathfrak{g}\rightarrow \CC_M\mathfrak{g}.
\end{eqnarray*}
We have $ \widetilde{\pi}_{N,M}(\CP_{\bm\Pi}) =\CP_{\bm\Pi'}$,
and the restriction of $\widetilde{\pi}_{N,M}$ gives an isomorphism $\fh_{\bm\Pi}\cong\fh_{\bm\Pi'}$.
This allows us to view $\chi$ as a character of $\mathcal{P}_{{\bm\Pi'}}$ and so one has the $\CC_N\mathfrak{g}$-module isomorphism 
\[
M_{\mathcal{P}_{\bm\Pi}}(\chi)
\cong
\pi_{N,M}^{-1}\bigl(M_{\mathcal{P}_{\bm\Pi'}}(\chi)\bigr).
\]
Then the irreducibility of $\CC_N\fg$-module $M_{\mathcal{P}_{\bm\Pi}}(\chi)$ reduces to the irreducibility of $\CC_M\fg$-module $M_{\mathcal{P}_{\bm\Pi'}}(\chi)$.

In the rest of this section, we assume that $\bm\Pi$ is a genuine $(N+2)$-partition of $\Pi$. 
For $i=1,\dots,N+1$, denote by $\bar{\mathfrak{g}}_i$ the semisimple Lie subalgebra of $\mathfrak{g}$ generated by $\mathfrak{g}_\alpha$ for $\alpha\in\bar{\Delta}_i$. Explicitly,
\begin{align}\label{bargi}
    \bar{\mathfrak{g}}_i=\bar{\mathfrak{h}}_i\oplus \sum_{\alpha\in\bar{\Delta}_i}\mathfrak{g}_\alpha,
\end{align}
where
\[
\bar{\mathfrak{h}}_i:=\sum_{j=0}^{i}\sum_{\alpha\in \Pi_j}\mathbb{C}\alpha^{\vee}
\]
is a Cartan subalgebra of $\bar{\mathfrak{g}}_i$ and $\bar{\Delta}_i$ is the root system of $\bar{\mathfrak{g}}_i$ with respect to $\bar{\mathfrak{h}}_i$. 
In addition, we have a triangular decomposition
\begin{align*}
\bar{\mathfrak{g}}_i=\bar\fn_i^{-}\oplus\bar{\mathfrak{h}}_i\oplus  \bar\fn_i^+
\end{align*}
with respect to the base $\overline{\Pi}_i$ of $\bar\Delta_i$, where $\bar\fn_i^\pm=\bar\fg_i\cap \fn^{\pm}$. 

Consider the induced $(i+1)$-partition
\[
\bm{\Pi}^i=(\Pi_0,\Pi_1,\dots,\Pi_i)
\]
of $\overline{\Pi}_i$. We obtain a parabolic subalgebra $\mathcal{P}_{\bm\Pi^i}$ of $\CC_{i-1}\bar{\mathfrak{g}}_i$ of type  $\bm{\Pi}^i$ with support 
\[
S_{\bm\Pi^i}
=((\bar\Delta_i^+\cup\{0\})\times\{0,1,\dots,i-1\})
\cup \bigcup_{k=0}^{i}\{(-\alpha,j)\mid \alpha\in \hat\Delta_k^+,\, k\le j\le i-1\}.
\]
Explicitly, 
\[
\mathcal{P}_{\bm\Pi^i}
=\CC_{i-1}\bar{\mathfrak{h}}_i\oplus  \CC_{i-1}\bar\fn_i^{+}\oplus \sum_{k=0}^{i}\sum_{\alpha\in \hat\Delta_k^+}\sum_{j=k}^{i-1}\C x_{-\alpha}(j).
\]

For every character $\mu$ of $\mathcal{P}_{\bm\Pi^i}$, 
we have the parabolically induced $\CC_{i-1}\bar{\mathfrak{g}}_i$-module
\[
M_{\bm{\Pi}^i}(\mu)
=\mathrm{Ind}_{\mathcal{P}_{\bm\Pi^i}}^{\CC_{i-1}\bar{\mathfrak{g}}_i}~\C_{\mu}.
\]
By Lemma \ref{le:pp},  characters of $\mathcal{P}_{\bm\Pi^i}$ can be identified with linear functionals on \[
\mathfrak{h}_{\bm\Pi^i}
:=\sum_{k=1}^i\sum_{\alpha\in \Pi_k}\sum_{j=0}^{k-1}\mathbb{C}\alpha^{\vee}(j).
\]
Write $\chi_i$ for the  restriction to $\chi$ on $\fh_{\bm\Pi^i}$ (viewed as a subspace of $\fh_{\bm\Pi}$). 

One easily verifies the following result.
\begin{prpt}\label{giiso}
The $\CC_{N}\bar{\fg}_i$-module  $\pi_{N,i-1}^{-1}(M_{\bm\Pi^i}(\chi_i))$ is isomorphic to the $\CC_{N}\bar{\fg}_i$-submodule of $M_{\mathcal{P}_{\bm\Pi}}(\chi)$ generated by $v_\chi$.
\end{prpt}

\subsection{Shapovalov  form on $M_{\CP_{\bm\Pi}}(\chi)$}\label{shap}
In this subsection, we introduce the Shapovalov  form on the parabolically induced $\CC_N\fg$-module $M_{\CP_{\bm\Pi}}(\chi)$, which plays a key role in the study of its irreducibility.

Let $\omega$ be the  anti-involution of $\fg$ determined by 
\[\omega(x_{\alpha})=x_{-\alpha}\quad\te{and}\quad\omega(h)=h\quad\textrm{for }\alpha\in\Delta^+,~h\in\fh.\]
This induces an anti-involution of $\CC_N\fg$, denoted by $\omega_N$, such that 
\[\omega_N(x_{\alpha}(n))=x_{-\alpha}(n)\quad\te{and}\quad\omega_N(h(n))=h(n)\quad\textrm{for }\alpha\in\Delta^+,~h\in\fh,~0\leq n\leq N.\]
Note that $\omega_N$ extends uniquely to an anti-involution of $\mathcal{U}(\CC_N\fg)$, still denoted by $\omega_N$. On the other hand, by the Poincaré-Birkhoff-Witt theorem,  there is a decomposition
\begin{align}\label{dec}
    \mathcal{U}(\CC_N\fg)=\mathcal{U}(\CC_N\fh)\oplus((\CC_N\fn^-)\mathcal{U}(\CC_N\fg)+\mathcal{U}(\CC_N\fg)(\CC_N\fn^+))
\end{align}
as a direct sum of $\mathcal{U}(\CC_N\fh)$-bimodules. 

Denote by
\begin{align*}
    {\bf{q}}:\ \mathcal{U}(\CC_N \fg)\to \mathcal{U}(\CC_N \fh)
\end{align*}
 the projection  map  associated with the decomposition \eqref{dec}. Define a bilinear form 
\[\textbf{F}:\mathcal{U}(\CC_N\fg)\times \mathcal{U}(\CC_N\fg)\to \mathcal{U}(\CC_N\fh)\]
by 
\begin{align}\label{deff}
    \textbf{F}(x,y)=\textbf{q}(\omega_N(x)y)\quad \textrm{for }x,y\in \mathcal{U}(\CC_N\fg).
\end{align}
One easily verifies that $\textbf{F}$ is a symmetric contravariant  bilinear form on $\mathcal{U}(\CC_N\fg)$, which is called the Shapovalov  form on $\mathcal{U}(\CC_N\fg)$ (cf.\,\cite{W2}).

Set $\widetilde{\chi}:=\chi|_{\CC_N\b}$, which is a character on $\CC_N\b$, or equivalently, a linear functional on $\CC_N\fh$. 
We further extend $\widetilde{\chi}$ uniquely to a unital algebra
homomorphism
\[
\widetilde{\chi}:\mathcal U(\CC_N\fh)\longrightarrow\C.
\]
Recall the usual Shapovalov form 
\[\mathbf B_{\widetilde{\chi}}^{\,\b}: M_{\CC_N\b}(\widetilde{\chi})\times M_{\CC_N\b}(\widetilde{\chi})\rightarrow \C\] 
on the Verma $\CC_N\fg$-module $M_{\CC_N\b}(\widetilde{\chi})$, which  is given by
\[
\mathbf B_{\widetilde{\chi}}^{\,\b}
(u_1v_{\widetilde{\chi}},u_2v_{\widetilde{\chi}})
=
\widetilde{\chi}\bigl(\mathbf F(u_1,u_2)\bigr),
\qquad
u_1,u_2\in\mathcal U(\CC_N\fn^-).
\]
The radical  of $\mathbf B_{\widetilde{\chi}}^{\,\b}$
is exactly the maximal proper submodule of
$M_{\CC_N\b}(\widetilde{\chi})$ (see \cite{W2}).
Since $M_{\CP_{\bm\Pi}}(\chi)$ is a quotient of $M_{\CC_N\b}(\widetilde{\chi})$, $\mathbf B_{\widetilde{\chi}}^{\,\b}$ descends uniquely to a
well-defined symmetric contravariant bilinear form
\[
\mathbf B_\chi:
M_{\CP_{\bm\Pi}}(\chi)\times M_{\CP_{\bm\Pi}}(\chi)
\longrightarrow\C,
\]
characterized by
\[
\mathbf B_\chi
(u_1v_\chi,u_2v_\chi)
=
\widetilde{\chi}\bigl(\mathbf F(u_1,u_2)\bigr),
\qquad
u_1,u_2\in\mathcal U(\CC_N\fn^-).
\]
We call $\mathbf B_\chi$ the Shapovalov form on
$M_{\CP_{\bm\Pi}}(\chi)$.
The radical $\operatorname{Rad}\mathbf B_\chi$ of $\mathbf B_\chi$ is also the maximal
proper submodule of $M_{\CP_{\bm\Pi}}(\chi)$. In particular,
$M_{\CP_{\bm\Pi}}(\chi)$ is irreducible if and only if
$\operatorname{Rad}\mathbf B_\chi=\{0\}$.

 

The study of the Shapovalov form on $M_{\CP_{\bm\Pi}}(\chi)$ will  reduce to the study of the restrictions 
\begin{align*}
    (\textbf{B}_{\chi})_{\alpha}: M_{\CP_{\bm\Pi}}(\chi)_{\chi|_{\fh}-\alpha}\times M_{\CP_{\bm\Pi}}(\chi)_{\chi|_{\fh}-\alpha}\to \C
\end{align*}
for all $\alpha\in Q^{+}$. Indeed,  the adjoint action of $\fh$ on $\mathcal{U}({\CC_N\fg})$ induces a $Q$-grading structure on $\mathcal{U}({\CC_N\fg})$ such that 
\begin{align*}
    \mathcal{U}({\CC_N\fg})=\bigoplus_{\alpha\in Q}\mathcal{U}({\CC_N\fg})_{\alpha},
\end{align*}
where \[\mathcal{U}({\CC_N\fg})_{\alpha}=\{x\in \mathcal{U}({\CC_N\fg})\mid [h,x]=\alpha(h)x \quad \te{for all}~ h \in \fh\}.\]
It follows that $\mathcal{U}({\CC_N\fh})\subseteq \mathcal{U}({\CC_N\fg})_{0}$ and the weight spaces
$ M_{\CP_{\bm\Pi}}(\chi)_{\chi|_{\fh}-\alpha} ~\text{and}~$
$M_{\CP_{\bm\Pi}}(\chi)_{\chi|_{\fh}-\beta}$
are orthogonal with respect to $\mathbf B_\chi$ whenever
$\alpha\neq\beta$. This yields 
\begin{align*}
    \mathrm{Rad\,}(\mathbf{B}_\chi)=\bigoplus_{\alpha\in Q^+}\mathrm{Rad\,}((\mathbf{B}_\chi)_{\alpha}), 
\end{align*}
where \[\mathrm{Rad\,}((\mathbf{B}_\chi)_{\alpha}):=\mathrm{Rad\,}\mathbf{B}_\chi\cap M_{\CP_{\bm\Pi}}(\chi)_{\chi|_{\fh}-\alpha}.\]

Consequently, we have 
\begin{lemt}\label{rad}
The $\CC_N\fg$-module $M_{\CP_{\bm\Pi}}(\chi)$ is reducible if and only if $\mathrm{Rad\,}((\mathbf{B}_\chi)_{\alpha})\neq \{0\}$
for some $\alpha\in Q^+$.
\end{lemt}

\subsection{Jantzen's irreducibility criterion}
As before, we identify $\bar{\mathfrak{g}}_1$ with $\CC_0\bar{\mathfrak{g}}_1$. In this subsection, we recall the irreducibility criterion for parabolically induced $\bar{\fg}_1$-modules established in \cite{J}.

Let $\mu$ be a linear functional on $\fh_{\bm\Pi^1}$.
Let $\bar\mu\in\bar\fh_1^*$ be the extension of $\mu$
satisfying
\[
\bar\mu|_{\fh_{\bm\Pi^1}}=\mu
\quad\text{and}\quad
\bar\mu(\alpha^\vee)=0
\quad\text{for all }\alpha\in\Pi_0.
\]
Consider the Borel subalgebra 
\[\mathcal{P}_{(\emptyset,\Pi_0\cup\Pi_1)}=\bar{\mathfrak{h}}_1\oplus \sum_{\alpha\in\bar{\Delta}_1^+}\mathfrak{g}_{\alpha}\] of $\bar{\mathfrak{g}}_1$.
The corresponding Verma $\bar\fg_1$-module of
highest weight $\bar\mu$ is
\[M_{\bar{\mathfrak{g}}_1}(\bar\mu):=\te{Ind}_{\CP_{(\emptyset,\Pi_0\cup\Pi_1)}}^{\bar{\mathfrak{g}}_1}\C_{\bar{\mu}}.\]
As usual, the character of  $M_{\bar{\mathfrak{g}}_1}(\bar{\mu})$ is given by
\begin{equation*}
\textrm{char}\,M_{\bar{\mathfrak{g}}_1}(\bar{\mu})=\sum_{\lambda\in\bar{\mathfrak{h}}_1^{\ast}}\dim\, (M_{\bar{\mathfrak{g}}_1}(\bar{\mu}))_{\lambda}\cdot e(\lambda).
\end{equation*} 

Write $\overline W_1$ for the Weyl group of the root system $\bar\Delta_1$ and write $W_{0}\subseteq \overline W_1$ for the subgroup generated by simple reflections $s_{\alpha}$ for $\alpha\in \Pi_0$. Set 
\[\rho:= \frac{1}{2}\sum_{\alpha\in\bar\Delta_1^+}\alpha.\]Recall the dot action of $\overline W_1$ on $(\bar{\mathfrak{h}}_1)^\ast$ is given by
\[w\cdot \gamma=w(\gamma+\rho)-\rho\quad \te{for }w\in \overline W_1,~\gamma\in (\bar{\mathfrak{h}}_1)^\ast. \]
For every $\gamma\in (\bar{\mathfrak{h}}_1)^\ast$, define
\[\theta(\gamma):=\sum_{w\in W_0}(-1)^{\ell(w)}\te{char}\,M_{\bar{\mathfrak{g}}_1}(w\cdot\gamma),\]
where $\ell(w)$ denotes the length of $w$ (see \cite{H}). In addition, we also set 
\[(\bar\Delta_1)_{\gamma}:=\{\alpha\in\hat\Delta_1^+\mid (\gamma+\rho)(\alpha^{\vee})\in\Z_{+}\}.\]
The following irreducibility criterion for $M_{\bm\Pi^1}(\mu)$,
together with its corollary, was established in \cite{J}.
\begin{prpt}\label{pr:0}
Let $\mu$ be a linear functional on $\fh_{\bm\Pi^1}$. Then the  parabolically induced ${\bar{\mathfrak{g}}_1}$-module $M_{\bm{\Pi}^1}(\mu)$ is irreducible if and only if the linear functional $\mu$ satisfies the condition that 
\begin{enumerate}
    \item [(J)] $\quad \sum_{\beta\in(\bar\Delta_1)_{\bar{\mu}}}\theta(s_{\beta}\cdot \bar{\mu})=0.$
\end{enumerate}
\end{prpt}

\begin{cort}\label{corg1}
Let $\mu$ be a linear functional on $\fh_{\bm\Pi^1}$. Assume that
 \[(\bar{\mu}+\rho)(\beta^{\vee})\neq 0\quad\textrm{for all  }\beta\in\hat\Delta_1.\] 
Then $M_{\bm{\Pi}^1}(\mu)$ is irreducible if and only if $(\bar\Delta_1)_{\bar{\mu}}=\emptyset$.
\end{cort}

\subsection{Irreducibility of parabolically induced modules}
In this subsection, let $\chi$ be a character of $\CP_{\bm\Pi}$. We now state the main result of this section, which gives an irreducibility criterion for $M_{\CP_{\bm\Pi}}(\chi)$.

\begin{dfnt}
Let $\chi$ be a  character of $\CP_{\bm\Pi}$. We say that 
$\chi$ is generic if either $N=0$, or $N\ge 1$ and \[\chi(\alpha^{\vee}(i-1))\neq 0\]
for all $\alpha\in \hat\Delta_i^+$ and $i=2,3,\dots,N+1$. We say that $\chi$ is nonsingular if it is generic and $\chi_1$ satisfies the condition $(J)$.
\end{dfnt}

The following is the main result of this section, whose proof will be presented in Section \ref{sec:5}.

 \begin{thm}\label{thm:main}
 Let $\bm\Pi$ be a genuine  $(N+2)$-partition of $\Pi$ 
 and let $\chi$ be a character of $\mathcal{P}_{\bm\Pi}$. Then $M_{\CP_{\bm\Pi}}(\chi)$ is irreducible if and only if  $\chi$ is nonsingular.
\end{thm}
As an immediate consequence of Corollary \ref{corg1} and Theorem \ref{thm:main}, we have the following result.

\begin{cort}
  Let $\bm\Pi$ be a genuine  $(N+2)$-partition of $\Pi$ and let $\chi$ be a  character of $\mathcal{P}_{\bm\Pi}$. Assume that $\chi$ is generic and 
  \[(\overline{\chi_1}+\rho)(\beta^{\vee})\neq 0\quad \textrm{ for all }\beta\in\hat\Delta_1.\] Then $M_{\CP_{\bm\Pi}}(\chi)$ is irreducible if and only if 
  $(\bar\Delta_1)_{\overline{\chi_1}}=\emptyset$.
\end{cort}

\subsection{Character formulas for irreducible highest weight $\CC_N\fg$-modules}\label{sec:chartak} Let $\chi$ be a linear functional on $\CC_N\fh$.
In this subsection, we give character formulas for almost all irreducible highest weight $\CC_N\fg$-modules $V_{\CC_N\b}(\chi)$. 

 If $\chi=0$, then $V_{\CC_N\b}(\chi)$ is the one-dimensional trivial $\CC_N\fg$-module and hence $\textrm{char}\,V_{\CC_N\b}(\chi)=e(0)$. Henceforth, assume that $\chi\neq 0$.
For every $\alpha\in \Pi$, define
\[\chi_\alpha=\max\{\,i+1\mid 0\le i\le N,\ 
\chi(\alpha^\vee(i))\neq0\,\}\]
with the convention that $\chi_\alpha=0$  if $\chi(\alpha^\vee(j))= 0$ for all $j=0,1,\dots,N$. Set 
\[
M=\max\{\chi_\alpha\mid \alpha\in \Pi\}-1.
\]
Then 
\begin{align}\label{Pichi}
    \bm\Pi(\chi)=(\Pi(\chi)_0,\Pi(\chi)_1,\dots,\Pi(\chi)_{M+1})
\end{align}
is a  genuine $(M+2)$-partition of $\Pi$, where 
\[\Pi(\chi)_{i}=\{\alpha\in\Pi\mid \chi_\alpha=i\}\quad \textrm{for }i=0,1,\dots,M+1.\]

Associated with the $(M+2)$-partition $\bm\Pi(\chi)$ of $\Pi$, one can define 
$\Delta_i(\chi)$, $\bar\Delta_i(\chi)$, $\hat\Delta_i(\chi)$, and $\hat\Delta_i(\chi)^\pm$ for $i=0,1,\dots,M+1$ analogously to  the corresponding objects in Subsection 3.2.
We also obtain the  (maximal) parabolic subalgebra of $\CC_{M}\fg$ of type $\bm\Pi(\chi)$ as follows: 
\[\CP_{\bm\Pi(\chi)}=\CC_M\b\oplus\sum_{j=0}^{M+1}\sum_{\alpha\in\hat\Delta_j(\chi)^+}\sum_{k=j}^{M}\C x_{-\alpha}(k).\]

As usual, write $\chi|_{\fh_{\bm\Pi(\chi)}}$ for the restriction of $\chi$ to \[\fh_{\bm\Pi(\chi)}=\sum_{i=1}^{M+1}\sum_{\alpha\in\Pi(\chi)_i}\sum_{j=0}^{i-1}\C\alpha^\vee(j).\] 

\begin{prpt}\label{pr:mirriso}
    Assume that $\chi|_{\fh_{\bm\Pi(\chi)}}$ is nonsingular. Then 
    \[V_{\CC_N\b}(\chi)\cong \pi_{N,M}^{-1}\left(M_{\CP_{\bm\Pi(\chi)}}(\chi|_{\fh_{\bm\Pi(\chi)}})\right)\]
    as $\CC_N\g$-modules. Consequently,
    \[ \mathrm{char}\,V_{\CC_N\b}(\chi)=\mathrm{char}\,M_{\CP_{\bm\Pi(\chi)}}(\chi|_{\fh_{\bm\Pi(\chi)}}).\]
\end{prpt}
\begin{proof}
The assertion follows from Lemma \ref{pullback} and Theorem \ref{thm:main}.
\end{proof}

\begin{thmt}\label{th:chartak}
Assume that $\chi|_{\fh_{\bm\Pi(\chi)}}$ is nonsingular. Then 
\begin{align*}
\mathrm{char}\, V_{\CC_N\b}(\chi)= e(\chi|_{\fh})\prod_{i=1}^{M+1}\prod_{\alpha\in\hat\Delta_i(\chi)^-}(1-e(\alpha))^{-i}.
\end{align*}
\end{thmt}
\begin{proof}
By Proposition \ref{pr:mirriso}, it suffices to compute the character formula of the parabolically induced  $\CC_M\fg$-module $M_{\CP_{\bm\Pi(\chi)}}(\chi|_{\fh_{\bm\Pi(\chi)}})$.
For $i=1,2,\dots,M+1$ and $j=0,1,\dots,i-1$, set
\begin{align*}
    \fg_{i,j}=\sum_{\alpha\in\hat\Delta_i(\chi)^-}\C x_{\alpha}(j) \subseteq \CC_M\fg.
\end{align*}
Since $\fg_{i,j}$ is stable under the adjoint action of $\fh$, its symmetric algebra $S(\fg_{i,j})$ is naturally an $\fh$-module.  By the Poincaré-Birkhoff-Witt theorem, there is an $\fh$-module isomorphism
\begin{align}\label{eq:chartak}
    M_{\CP_{\bm\Pi(\chi)}}(\chi|_{\fh_{\bm\Pi(\chi)}})\cong \C_{\chi|_\fh}\otimes(\bigotimes_{i=1}^{M+1}\bigotimes_{j=0}^{i-1}~S(\fg_{i,j})).
\end{align}
Hence
\begin{align*}
   \textrm{char}\, M_{\CP_{\bm\Pi(\chi)}}(\chi|_{\fh_{\bm\Pi(\chi)}})=e(\chi|_{\fh})\prod_{i=1}^{M+1}\prod_{j=0}^{i-1}\textrm{char}\,S(\fg_{i,j}).
\end{align*} 

For each $i=1,\dots,M+1$ and $j=0,\dots,i-1$, the
$\fh$-weight of $x_\alpha(j)$ is $\alpha$. Therefore,
\[
\operatorname{char}S(\fg_{i,j})
=
\prod_{\alpha\in\hat\Delta_i(\chi)^-}
(1-e(\alpha))^{-1}.
\]
It follows that
\begin{align*}
\operatorname{char}V_{\CC_N\b}(\chi)
&=
e(\chi|_\fh)
\prod_{i=1}^{M+1}
\prod_{j=0}^{i-1}
\prod_{\alpha\in\hat\Delta_i(\chi)^-}
(1-e(\alpha))^{-1}\\
&=
e(\chi|_\fh)
\prod_{i=1}^{M+1}
\prod_{\alpha\in\hat\Delta_i(\chi)^-}
(1-e(\alpha))^{-i}.
\end{align*}
\end{proof}

\section{Proof of Theorem \ref{thm:main}: the irreducibility criterion}\label{sec:5}
In this section, let $\bm\Pi$ be a genuine  $(N+2)$-partition of $\Pi$
and let $\chi$ be a  character of $\mathcal{P}_{\bm\Pi}$, as in Theorem \ref{thm:main}.
\subsection{A basis of $M_{\CP_{\bm\Pi}}(\chi)$} 
In this subsection, we construct a basis of $M_{\CP_{\bm\Pi}}(\chi)$. Recall from Remark \ref{remdec} that 
\[\CC_N\fg=\fn_{\bm\Pi}^-\oplus\CP_{\bm\Pi},\]
where \begin{align*}
\fn_{\bm\Pi}^-:=\fn_{\bm\Pi,\widetilde{\Delta}_{\CP_{\bm\Pi}}}^-=\sum_{i=1}^{N+1}\sum_{\alpha\in \hat\Delta_i^-}\sum_{j=0}^{i-1}\C x_{\alpha}(j).
\end{align*}
For convenience, we introduce the  notation
\[
\tilde\Delta^-:=\bigcup_{i =1}^{N+1}\tilde\Delta_i^-\qquad (\tilde\Delta_i^-:=\hat\Delta_i^-\times \{0,1,\dots,i-1\}).
\]
Then the set
\[\mathcal{B}:=\{x_{\alpha}(j)\in \CC_N\fg\mid (\alpha,j)\in \tilde\Delta^-\}\] forms a basis of $\fn_{\bm\Pi}^-$.

To apply the Poincaré-Birkhoff-Witt theorem, we  introduce a total order on this basis $\mathcal{B}$.
Recall that there is a partial order $``\succeq"$ on $\Delta$ such that
\begin{align*}
    \alpha\succeq\beta \iff \alpha-\beta\in Q^+.
\end{align*}
Fix a total order $\geq_\Delta$ on $\Delta$ extending $``\succeq"$. We define a total order $``\geq "$ on $\mathcal{B}$ by declaring that  $x_{\alpha_1}(j_1)>x_{\alpha_2}(j_2)$, where  $(\alpha_1,j_1)\in \tilde\Delta_{i_1}^-,~(\alpha_2,j_2)\in \tilde\Delta_{i_2}^-$,  if  one of the following conditions holds:
\begin{enumerate}
    \item $i_1>i_2$;
    \item $i_1=i_2$ and $j_1>j_2$; and
    \item $i_1=i_2$, $j_1=j_2$, and $\alpha_1>_{\Delta}\alpha_2$.
\end{enumerate}

Let \[\mathcal{M}=\{\{m_{\alpha,d}\}_{(\alpha,d)\in \tilde\Delta^-}\mid m_{\alpha,d}\in \BN \text{ for }(\alpha,d)\in \tilde\Delta^-\}.\]
For every $M=\{m_{\alpha,d}\}_{(\alpha,d)\in \tilde\Delta^-}\in\mathcal{M}$, we write
\begin{align}\label{eq}
    X_M:=\prod_{(\alpha,d)\in\tilde\Delta^-}x_{\alpha}(d)^{m_{\alpha,d}},
\end{align}
where the product is taken according to the total order  $\geq$ on $\mathcal{B}$. That is, if 
\[
x_{\alpha_1}(d_1)>x_{\alpha_2}(d_2)>\cdots>x_{\alpha_k}(d_k)
\]
are the elements of $\mathcal{B}$ appearing with nonzero multiplicities in
$M$, then
\[
X_M=
x_{\alpha_1}(d_1)^{m_{\alpha_1,d_1}}
x_{\alpha_2}(d_2)^{m_{\alpha_2,d_2}}
\cdots
x_{\alpha_k}(d_k)^{m_{\alpha_k,d_k}}.
\]
By the Poincar\'e-Birkhoff-Witt theorem and the vector space
decomposition $\CC_N\fg=\fn_{\bm\Pi}^-\oplus\CP_{\bm\Pi},$
the set
\begin{align}\label{eq:basispara}
\{X_Mv_\chi\mid M\in\mathcal M\}
\end{align}
forms a basis of $M_{\CP_{\bm\Pi}}(\chi)$.

For later use, we introduce some notation.  For every $M=\{m_{\alpha,d}\}_{(\alpha,d)\in \tilde\Delta^-}\in\mathcal{M}$, we set
\begin{align*}
    \Delta(M)=\{\alpha\in\Delta^-
\mid
m_{\alpha,d}\neq0\text{ for some }d\}\quad \textrm{and}\quad |M|=\sum_{(\alpha,d)\in \tilde\Delta^-}m_{\alpha,d}\in\BN.
\end{align*}
 In addition,  for every $(\beta,d')\in \tilde\Delta^-$ with $m_{\beta,d'}>0$, define
\[
M^{(\beta,d')} = \{m_{\alpha,d}^{(\beta,d')}\}_{(\alpha,d)\in\tilde\Delta^-}
\in\mathcal M
\]
by
\[
m_{\alpha,d}^{(\beta,d')}=m_{\alpha,d}-\delta_{\alpha,\beta}\delta_{d,d'}.
\]
Thus $M^{(\beta,d')}$ is obtained from $M$ by removing one copy of
$(\beta,d')$.
For $i=1,2,\dots,N+1$, set 
\[ \mathcal{M}_i=\{\{m_{\alpha,d}\}_{(\alpha,d)\in \tilde\Delta^-}\in \mathcal{M}\mid  m_{\alpha,d}=0\textrm{ unless }(\alpha,d)\in\tilde{\Delta}_i^-\}.\]
For every
\[
M=\{m_{\alpha,d}\}_{(\alpha,d)\in\tilde\Delta^-}\in\mathcal M
\]
and $i=1,2,\dots,N+1$, let $M_i\in\mathcal M_i$ be defined by 
\[
(M_i)_{\alpha,d} = \begin{cases}
m_{\alpha,d},
&(\alpha,d)\in\tilde\Delta_i^-,\\
0,&(\alpha,d)\notin\tilde\Delta_i^-.
\end{cases}
\]
By the definition of the total order on $\mathcal B$, we then have
\[
X_M = X_{M_{N+1}}X_{M_N}\cdots X_{M_1}.
\]

\subsection{Sufficient condition for irreducibility of $M_{\CP_{\bm\Pi}}(\chi)$}
In this subsection, we fix an integer $i=1,2,\dots,N+1$. Recall the subalgebra $\bar{\mathfrak{g}}_i$ of $\fg$ (see \eqref{bargi}). By taking restriction, $M_{\CP_{\bm\Pi}}(\chi)$ is a $\CC_N\bar{\mathfrak{g}}_i$-module, and we 
write
\[W_\chi(\bar{\mathfrak{g}}_i):=\mathcal{U}(\CC_N\bar{\mathfrak{g}}_i).v_{\chi}\]
for the $\CC_N\bar{\mathfrak{g}}_i$-submodule 
of $M_{\CP_{\bm\Pi}}(\chi)$ generated by $v_{\chi}$. 

Recall that $\chi_i$ is the restriction of $\chi$ to $\fh_{\bm\Pi^i}=\sum_{k=1}^i\sum_{\alpha\in \Pi_k}\sum_{j=0}^{k-1}\mathbb{C}\alpha^{\vee}(j)$. 
The following result concerns the irreducibility of the $\CC_N\bar{\mathfrak{g}}_i$-module $W_\chi(\bar{\mathfrak{g}}_i)$. 
\begin{thmt}\label{th:mainirr}
    The $\CC_N\bar{\mathfrak{g}}_i$-module $W_\chi(\bar{\mathfrak{g}}_i)$ is irreducible if $\chi_i$ is nonsingular. In particular, the $\CC_N\g$-module $M_{\CP_{\bm\Pi}}(\chi)$ is irreducible if $\chi$ is nonsingular.
 \end{thmt}
The rest of this subsection is devoted to a proof of Theorem \ref{th:mainirr}. The case $i=1$ follows from Lemma \ref{pullback} (with $\eta=\pi_{N,0}$ and $V=M_{\bm\Pi^1}(\chi_1)$), Proposition \ref{giiso}, and Proposition \ref{pr:0}. 
Now we assume that $i\geq 2$.  For $0 \le s \le n$, set
\begin{align*}
    \mathbb{I}_{n,s}=\{\bm{j}=(j_1,j_2,\dots ,j_s)\in\mathbb{N}^s\mid 1\leq j_1<j_2<\cdots< j_s\leq n\},
\end{align*}
where $\mathbb{I}_{n,0}:=\{\emptyset\}.$
For every $\bm{j}=(j_1,\dots ,j_s)\in\mathbb{I}_{n,s}$, write $\hat{\bm{j}}=(j_1',\dots,j_{n-s}')\in \mathbb{I}_{n,n-s}$ such that
\[\{j_1,j_2,\dots,j_s,j_1',j_2',\dots, j_{n-s}'\}=\{1,2,\dots,n\}.\]

 We shall need the following technical fact (see \cite[Eq. (4.6)]{CGLW}).
\begin{lemt}\label{commu}
Let $\mathfrak{l}$ be a Lie algebra and let $u,u'\in \mathcal{U}{(\mathfrak{l})}$. Suppose that 
\begin{align*}
   u'=y_1y_2\cdots y_n
\end{align*}
with each $y_j\in \mathfrak{l}$, then one has that 
\begin{align*}
[u,u']=\sum_{s=1}^n(-1)^s\sum_{\substack{\bm{j}=(j_1,j_2,\dots,j_s)\in\mathbb{I}_{n,s}}}
    y_{j_1'}\cdots y_{j_{n-s}'}\mathrm{ad\,}y_{j_s}\cdots \mathrm{ad\,}y_{j_1}(u).
\end{align*}
\end{lemt}
 Set 
 \[\CC_N\bar{\g}_{\geq i}=\textrm{Span}_{\BC}\{x(n)\mid x\in \bar{\mathfrak{g}}_i,~n\ge i\},\]
 a subalgebra of $\CC_N\fg$.

\begin{lemt}\label{le:basiccomm}
Let $v\in W_\chi(\bar{\mathfrak{g}}_{i-1})$ and let $\alpha\in\hat\Delta_i^+$. Then we have 
\begin{enumerate}
    \item $\CC_N\bar\g_{\geq i-1}.v=0$, 
    \item $\CC_N\fg_{\alpha}.v=0$, and 
    \item $\alpha^{\vee}(i-1).v=\chi(\alpha^\vee(i-1))v$.
\end{enumerate}
\end{lemt}
\begin{proof}
By linearity, it suffices to consider $v=u'v_\chi$, where 
\[u'=X_{M_{i-1}}X_{M_{i-2}}\cdots X_{M_1}\] for some $M_j\in\mathcal{M}_j$, $1\leq j\leq i-1$. 
Note that the assertion $(3)$ follows from assertion $(1)$ and Lemma $\ref{commu}$ with $u=\alpha^\vee(i-1)$. Then it suffices to prove the first two assertions.
For the first assertion, 
since $[\CC_N\bar\g_{\geq i-1},\CC_N\bar\g_{i-1}]\subseteq \CC_N\bar\g_{\geq i-1}$ and $\CC_N\bar\g_{\geq i-1}$ acts trivially on $v_\chi$, assertion $(1)$ follows immediately from Lemma \ref{commu} (with $u\in \CC_N\bar\g_{\geq i-1}$).

Every nonzero iterated commutator of $x_\alpha(k)$ with root
vectors occurring in $u'$ is again a root vector, and its root
belongs to $\hat\Delta_i^+$. Hence it annihilates $v_\chi$. Lemma \ref{commu}
therefore gives
\[
[x_\alpha(k),u']v_\chi=0.
\]
Since $x_\alpha(k)v_\chi=0$, we obtain
\[
x_\alpha(k)v=0.
\]
\end{proof}

For every $k=0,1,\dots,i-1$, we introduce the set
\begin{align*}
\mathcal{M}_{i,k}=\{\{m_{\alpha,d}\}_{(\alpha,d)\in \tilde\Delta^-}\in \mathcal{M}_i\mid  m_{\alpha,d}=0\textrm{ unless }d=k\}.
\end{align*}
With this notation, we have that $W_\chi(\bar{\mathfrak{g}}_i)$ is spanned by the vectors of the form
\begin{align}\label{spanele}
    X_{M^{i-1}}X_{M^{i-2}}\cdots X_{M^{0}}.w
\end{align}
for $w\in W_\chi(\bar{\mathfrak{g}}_{i-1})$, $M^{k}\in\mathcal{M}_{i,k}$ with $k=0,1,\dots,i-1$.
For a fixed $k=0,1,\dots,i-1$, one obtains an $\BN$-grading structure on $W_\chi(\bar{\mathfrak{g}}_i)$ given by $|M^k|$. That is,
\begin{align}\label{gradj}
W_\chi(\bar{\mathfrak{g}}_i)=\bigoplus_{n\in\BN}W_\chi(\bar{\mathfrak{g}}_i)_{(k,n)},
\end{align}
where $W_\chi(\bar{\mathfrak{g}}_i)_{(k,n)}$ is spanned by the elements 
\eqref{spanele} with $|M^k|=n$. 

\begin{lemt}\label{fmaj}
    Let $M\in\mathcal{M}_{i,k}$,  $w\in W_\chi(\bar{\mathfrak{g}}_{i-1})$, and  $-\alpha\in \Delta(M)$. Then we have 
    \begin{align*}
        x_{\alpha}(i-1-k)X_M.w=\ &m_{-\alpha,k}\chi(\alpha^\vee(i-1))X_{M^{(-\alpha,k)}}.w\notag\\
        &+\sum_{\substack{\beta\in \Delta(M),\\ \alpha+\beta\in \hat\Delta_i^-}}c_{\beta,\alpha}x_{\alpha+\beta}(i-1)X_{M^{(\beta,k)}}.w+\delta_{k,0}w',
    \end{align*}
     where $c_{\beta,\alpha}\in\C$ and $w'\in \sum_{n=0}^{|M|-2}W_\chi(\bar{\mathfrak{g}}_i)_{(k,n)}$.
\end{lemt}
\begin{proof}
For convenience, set $n=|M|$ and write $X_M=x_{\alpha_1}(k)x_{\alpha_2}(k)\cdots x_{\alpha_n}(k)$ for some $\alpha_1,\alpha_2,\dots,\alpha_n\in\hat\Delta_i^-$ with $\alpha_1\geq_\Delta \alpha_2\geq_\Delta\cdots\geq_\Delta \alpha_n$. In view of Lemma \ref{commu} and the second assertion of Lemma \ref{le:basiccomm}, we have 
\begin{align}\label{xMf}
    x_{\alpha}(i-1-k)X_M.w=\sum_{s=1}^n(-1)^s\sum_{\substack{\bm{j}=(j_1,j_2,\dots,j_s)\in\mathbb{I}_{n,s}}}f(\bm{j}).w,
\end{align}
where 
\begin{align*}
    f(\bm{j})= x_{\alpha_{j_1'}}(k)\cdots x_{\alpha_{j_{n-s}'}}(k)\mathrm{ad\,}x_{\alpha_{j_s}}(k)\cdots \mathrm{ad\,}x_{\alpha_{j_1}}(k)(x_{\alpha}(i-1-k)).
\end{align*}
Assume that $\bm{j}\in \mathbb{I}_{n,s}$ for $s\geq 2$. It is obvious  that 
\begin{align}\label{eqs2}
    f(\bm{j}).w\in \sum_{n=0}^{|M|-2}W_\chi(\bar{\mathfrak{g}}_i)_{(k,n)}.
\end{align}
 In particular, when $k>0$, one has \[\mathrm{ad\,}x_{\alpha_{j_s}}(k)\cdots \mathrm{ad\,}x_{\alpha_{j_1}}(k)(x_{\alpha}(i-1-k))\in \CC_N\bar\fg_{\geq i}\] and hence
\begin{align}\label{eqs20}
    f(\bm{j}).w=0
\end{align}
by Lemma \ref{le:basiccomm} (1).

It remains to consider the case where  $\bm{j}=j_1\in \mathbb{I}_{n,1}$. In this case, we have
\begin{align}\label{eqexp=}
    f(\bm{j}).w=\begin{cases}
        0,&\mbox{if }\alpha_{j_1}+\alpha\notin \hat\Delta_i^-\cup\{0\},\\
        -\chi(\alpha^\vee(i-1))X_{M^{(-\alpha,k)}}.w, &\mbox{if }\alpha_{j_1}+\alpha=0
    \end{cases}
\end{align}
by Lemma \ref{le:basiccomm}. When $\alpha_{j_1}+\alpha\in \hat\Delta_i^-$, we obtain 
\begin{align}\label{eqs1}
    f(\bm{j}).w=\ &c_{\alpha_{j_1}}X_{M^{(\alpha_{j_1},k)}}x_{\alpha+\alpha_{j_1}}(i-1).w\notag\\
    =\ &c_{\alpha_{j_1}}x_{\alpha+\alpha_{j_1}}(i-1)X_{M^{(\alpha_{j_1},k)}}.w+\delta_{k,0}w'',
\end{align}
where $w''\in \sum_{n=0}^{|M|-2}W_\chi(\bar{\mathfrak{g}}_i)_{(k,n)}$ and $c_{\alpha_{j_1}}\in\mathbb{C}^{\times}$. Here, we have used Lemma \ref{le:basiccomm} (1) in the last equality. Therefore, combining \eqref{eqs2}-\eqref{eqs1}, the lemma follows immediately.
\end{proof}

\begin{prpt}\label{prop:destoi-1}
    Let $i>1$ and let $v\in W_\chi(\bar{\mathfrak{g}}_i)\backslash W_\chi(\bar{\mathfrak{g}}_{i-1})$. Assume that $\chi_i$ is nonsingular. Then there exists $(-\alpha,k)\in \tilde\Delta_i^-$ such that 
    \[x_{\alpha}(k).v\neq 0.\]
\end{prpt}

\begin{proof}
Choose the smallest integer $k\in\{0,1,\dots,i-1\}$ such that
\[
v=\sum_{M\in\mathcal A}u_MX_Mw_M,
\]
where $\mathcal A\subseteq\mathcal M_{i,k}$ is finite with
$|M|>0$ for some $M\in\mathcal A$, $0\neq w_M\in W_\chi(\bar{\mathfrak g}_{i-1})$, and
\[
0\neq u_M\in
\operatorname{span}_{\C}
\left\{
X_{M^{i-1}}\cdots X_{M^{k+1}}
\mid
M^j\in\mathcal M_{i,j},\ k<j\le i-1
\right\}.
\]

Consider the $\BN$-grading of $W_{\chi}(\bar{\fg}_i)$ defined in
\eqref{gradj} with respect to the above choice of $k$. For
$u\in W_\chi(\bar{\mathfrak{g}}_i)$ and $n\in\BN$, denote by
$u_{(n)}$ its projection onto
$W_\chi(\bar{\mathfrak{g}}_i)_{(k,n)}$. Set
\[
D=\{n>0\mid v_{(n)}\neq0\},
\]
and define
\[
d=
\begin{cases}
\max D,& k=0,\\
\min D,& k>0.
\end{cases}
\]

Choose
\[
\alpha_d\in
\{\beta\in\Delta(M)\mid M\in\mathcal A\te{ and }|M|=d\}
\]
minimal with respect to the partial order $``\succeq"$. If
\[
-\alpha_d+\beta\in\hat\Delta_i^-,
\]
then
\[
\alpha_d-\beta\in Q^+\setminus\{0\},
\]
so that $\alpha_d\succ\beta$, contradicting the minimality of $\alpha_d$. Since every root vector occurring in $u_M$ has $t$-degree strictly greater than $k$, every nonzero commutator term arising from $ [x_{-\alpha_d}(i-1-k),u_M] $
has $t$-degree at least $i$. Hence, by Lemma \ref{le:basiccomm}(1), \[
[x_{-\alpha_d}(i-1-k),u_M]X_Mw_M=0.
\]
Therefore,
\[
x_{-\alpha_d}(i-1-k)v = \sum_{M\in\mathcal A}
u_Mx_{-\alpha_d}(i-1-k)X_Mw_M.
\]
When $k=0$, we have $d=\max D$, and the remainder term in
Lemma \ref{fmaj} has degree at most $d-2$. When $k>0$, the
remainder term vanishes, and since $d=\min D$, components of degree greater than $d$ cannot contribute to degree $d-1$. Therefore,
\begin{align}\label{d-1}
(x_{-\alpha_d}(i-1-k)v)_{(d-1)}= \sum_{\substack{M\in\mathcal A\\|M|=d}}
m_{\alpha_d,k}\chi((-\alpha_d)^\vee(i-1))
u_MX_{M^{(\alpha_d,k)}}w_M.
\end{align}

Since $\chi_i$ is nonsingular and
$-\alpha_d\in\hat\Delta_i^+$,
\[
\chi((-\alpha_d)^\vee(i-1))\neq0.
\]
By the choice of $\alpha_d$, we have $m_{\alpha_d,k}>0$ for some
$M\in\mathcal A$ with $|M|=d$. Moreover, the map
\[
M\longmapsto M^{(\alpha_d,k)}
\]
is injective on such $M$, and the PBW decomposition in
Subsection 5.1 shows that the corresponding summands cannot cancel.
Hence the right-hand side of \eqref{d-1} is nonzero, and consequently
\[
x_{-\alpha_d}(i-1-k).v\neq0.
\]
Since
\[
(\alpha_d,i-1-k)\in\tilde\Delta_i^-,
\]
the assertion follows.
\end{proof}
\noindent\textbf{Proof of  Theorem \ref{th:mainirr}:} We now complete the proof of Theorem \ref{th:mainirr} by induction on
$i$. The case $i=1$ has already been proved. Let $i>1$, assume that
$\chi_i$ is nonsingular, and let $S$ be a nonzero
$\CC_N\bar{\mathfrak g}_i$-submodule of
$W_\chi(\bar{\mathfrak g}_i)$. Choose a nonzero weight vector
$v\in S$ such that $\mathrm{wt\,}v'-\mathrm{wt\,}v\notin Q^+$ for all  nonzero weight vectors $v'\in S$ with $\mathrm{wt\,}v'\neq \mathrm{wt\,}v$.
If $v\notin W_\chi(\bar{\mathfrak g}_{i-1})$, 
then by Proposition \ref{prop:destoi-1} there exists $(-\alpha,k)\in\tilde\Delta_i^-$ such that
\[
x_\alpha(k)v\neq0.
\]
But $\mathrm{wt\,}(x_\alpha(k)v)-\mathrm{wt\,}v=\alpha\in Q^+$, 
a contradiction. Hence
\[
0\neq v\in
S\cap W_\chi(\bar{\mathfrak g}_{i-1}).
\]
Since $\chi_i$ is nonsingular, so is $\chi_{i-1}$. By the induction
hypothesis, $W_\chi(\bar{\mathfrak g}_{i-1})$ is irreducible.
Therefore,
\[
W_\chi(\bar{\mathfrak g}_{i-1})\subseteq S,
\]
and in particular $v_\chi\in S$, as required.

\subsection{Determinant of the Shapovalov form}\label{sec:5.3}
In this subsection, assume that $N\ge1$ and that $\chi_N$ is
nonsingular. Hence the $\bar\fg_N$-module $W_\chi(\bar\fg_N)$ is irreducible by Theorem \ref{th:mainirr}. By Lemma \ref{rad}, the restriction of  $\mathbf B_{\chi}$ to $W_{\chi}(\bar\fg_{N})$ is  nondegenerate. 
Therefore,  there exists a basis
\[
\{v_j:=u_jv_{\chi}\}_{j\in I}
\]
of $W_{\chi}(\bar\fg_N)$ consisting of weight vectors such that
\begin{equation*}
\mathbf B_{\chi}(v_{j_1},v_{j_2})
=\delta_{j_1,j_2},\quad j_1,j_2\in I.
\end{equation*}
Here, $I$ is an index set and 
each $u_j$ 
is a $\C$-linear combination of the elements $X_M$, where $M=\{m_{\alpha,d}\}_{(\alpha,d)\in\tilde\Delta^-}\in\mathcal{M}$ satisfies  $m_{\alpha,d}=0$ for all $(\alpha,d)\in\tilde\Delta_{N+1}^-$.

Note that $\fh_{\bm\Pi}=\fh_{N+1}\oplus\fh_{\bm\Pi^N}$, where \[\fh_{N+1}:=\sum_{d=0}^{N}\sum_{\alpha\in\Pi_{N+1}}\C\alpha^{\vee}(d).\]
 For every  $\mu\in \fh_{N+1}^\ast$,
let $\tilde\mu\in\fh_{\bm\Pi}^*$ be the unique linear functional
such that
\[
\tilde\mu|_{\fh_{N+1}}=\mu
\qquad\text{and}\qquad
\tilde\mu|_{\fh_{\bm\Pi^N}}
=\chi|_{\fh_{\bm\Pi^N}}.
\]
Then  $W_{\tilde\mu}(\bar\fg_{N})$ is an irreducible $\CC_N\bar\fg_N$-module isomorphic to $W_\chi(\bar\fg_N)$, and therefore 
\begin{equation}\label{eq:deltajij2}
\mathbf B_{\tilde\mu}(u_{j_1}v_{\tilde{\mu}},u_{j_2}v_{\tilde{\mu}})
=\delta_{j_1,j_2}.
\end{equation}


Recall from \cite{GW} the notion of polynomial functions on a finite-dimensional complex vector space. %
We extend $\tilde\mu$ uniquely to an algebra homomorphism
$\mathcal U(\CC_N\fh)\to\C$, still denoted by $\tilde\mu$. For every $h\in \mathcal{U}(\CC_N\fh)$, define the polynomial function $p_h:\fh_{N+1}^*\to\C$
by
\begin{align*}
    p_h(\mu)=\tilde\mu(h)\quad \textrm{for all }\mu\in \fh_{N+1}^\ast.
\end{align*}
It follows that 
\[\mathcal{A}(\fh_{N+1}^{\ast}):=\C[p_{\alpha^{\vee}(k)}\mid \alpha\in\Pi_{N+1},~0\leq k\leq N]\]
is precisely the algebra of all polynomial functions on $\fh_{N+1}^\ast$. In what follows, we will view the determinant of the Shapovalov form as a polynomial function on $\fh_{N+1}^\ast$ and express it in terms of $p_{\alpha^{\vee}(k)}$ for $\alpha\in\Pi_{N+1}$ and $0\leq k\leq N$. 

Fix $\gamma \in Q^{+}$ and set 
\[\tilde{I}_\gamma:=\{(M,i)\in \mathcal{M}_{N+1}\times I\mid \mathrm{wt\,}(X_Mu_iv_{\chi})=\chi|_{\fh}-\gamma\}.\]
Fix a total ordering on $\tilde{I}_\gamma$. For every $\mu\in \fh_{N+1}^\ast$, the ordered set
\[
\{\tilde{X}_M^i:=X_Mu_iv_{\tilde\mu}\mid (M,i)\in \tilde{I}_\gamma\},
\]
indexed according to this ordering, forms an ordered basis of
$M_{\CP_{\bm\Pi}}(\tilde{\mu})_{\tilde{\mu}|_{\fh}-\gamma}$. 
Let  $\mathrm{det}_{\gamma}(\mu)$ denote the determinant of $(\mathbf{B}_{\tilde{\mu}})_\gamma$ with respect to this basis. Then we obtain a polynomial function
\[\mathrm{det}_{\gamma}:\ \fh_{N+1}^\ast\to\C,\quad \mu\mapsto \mathrm{det}_{\gamma}(\mu).\]

The following result expresses $\mathrm{det}_\gamma$ in terms of the generators of $\mathcal{A}(\fh_{N+1}^\ast)$.

\begin{prpt}\label{det}
    One has that 
    \begin{align*}
\mathrm{det}_\gamma=k\prod_{\alpha\in\hat\Delta_{N+1}^+}p_{\alpha^{\vee}(N)}^{k_{\alpha}},
    \end{align*}
    where $k\in\C^{\times}$ and $k_{\alpha}\in\BN$.
\end{prpt}
\begin{proof}
Assume that
\[
p_{\alpha^\vee(N)}(\mu)=\tilde\mu(\alpha^\vee(N))
\neq0
\qquad
\text{for all }\alpha\in\hat\Delta_{N+1}^+.
\]
Since the restriction of $\tilde\mu$ to $\fh_{\bm\Pi^N}$ equals $\chi_N$, which is nonsingular, it follows that $\tilde\mu$ is
nonsingular. Hence Theorem \ref{th:mainirr} and Lemma \ref{rad} give $\mathrm{det}_\gamma(\mu)\neq 0$.
Write $V$ for the zero set of $\mathrm{det}_\gamma$. This yields
\[V\subseteq\{\mu\in \fh_{N+1}^{\ast}\mid  p_{\alpha^\vee(N)}(\mu)=0 \textrm{ for some  }\alpha\in\hat\Delta_{N+1}^+\}\]
and hence
\begin{align}\label{eqzerov}
\prod_{\alpha\in\hat\Delta_{N+1}^+}p_{\alpha^{\vee}(N)}\in I(V):=\{g\in \mathcal{A}(\fh_{N+1}^\ast)\mid g|_{V}=0\}.
\end{align}

 By Hilbert's Nullstellensatz, we have
    \[I(V)=\sqrt{(\mathrm{det}_\gamma)},\]
    where $(\det_\gamma)$ denotes the principal ideal generated by
$\det_\gamma$ and $\sqrt{(\mathrm{det}_\gamma)}$ denotes its radical. Therefore, 
\eqref{eqzerov} implies that 
\[\mathrm{det}_\gamma\quad\textrm{divides}\quad\left(\prod_{\alpha\in\hat\Delta_{N+1}^+}p_{\alpha^{\vee}(N)}\right)^m\]
for some $m\in\Z_{+}$. Note that $\mathcal{A}(\fh_{N+1}^\ast)$ is a unique factorization domain
and  each
$p_{\alpha^\vee(N)}$
is irreducible. Then the assertion follows.

\end{proof}

Next, we  determine the exponents of $p_{\alpha^\vee(N)}$ in $\det_\gamma$. To this end, we introduce  $\BN$-graded algebra structures on  $\mathcal{A}(\fh_{N+1}^{\ast})$ and $\mathcal{U}(\CC_N\fh)$.
Equip  $\mathcal{A}(\fh_{N+1}^{\ast})$ with an $\BN$-graded algebra structure 
\begin{align}\label{gradpoly}
    \mathcal{A}(\fh_{N+1}^{\ast})=\bigoplus_{n\in\BN}\mathcal{A}(\fh_{N+1}^{\ast})_{(n)}
\end{align}
determined by 
\begin{align*}
    p_{\alpha^{\vee}(N)}\in \mathcal{A}(\fh_{N+1}^{\ast})_{(1)}\quad \textrm{and}\quad p_{\alpha^{\vee}(k)}\in \mathcal{A}(\fh_{N+1}^{\ast})_{(0)}
\end{align*}
for all $\alpha\in\Pi_{N+1}, 0\leq k\leq N-1$. Since 
    $\tilde{\mu}(\beta^\vee(N))=\chi(\beta^\vee(N))=0$ for all $\beta\in\overline{\Pi}_N$,
it follows that 
\begin{align}\label{homo}
     p_{\alpha^{\vee}(N)}\in \mathcal{A}(\fh_{N+1}^{\ast})_{(1)}\quad \textrm{for all }\alpha\in\hat\Delta_{N+1}^+.
\end{align}
For $f\in  \mathcal{A}(\fh_{N+1}^{\ast})$, we write $\deg_N f$ for its degree with respect to the grading \eqref{gradpoly}. 

Similarly, equip $\mathcal{U}(\CC_N\fh)$ with the $\BN$-graded algebra structure 
\begin{align}\label{eq:gradh}
    \mathcal{U}(\CC_N\fh)=\bigoplus_{n\in\BN} \mathcal{U}(\CC_N\fh)_{(n)}
\end{align}
determined by
\begin{align*}
    \alpha^{\vee}(N)\in \mathcal{U}(\CC_N\fh)_{(1)},\quad\beta^{\vee}(N) \in \mathcal{U}(\CC_N\fh)_{(0)},\quad\textrm{and}\quad h(k)\in \mathcal{U}(\CC_N\fh)_{(0)}
\end{align*}
for all  $\alpha\in\Pi_{N+1},\beta\in\overline{\Pi}_{N},~h\in\fh$, and $0\leq k\leq N-1$. For
$u\in\mathcal{U}(\CC_N\fh)$, we write $\deg_Nu$ for its degree with respect to this grading. By convention, $\deg_N0=-1$ in both $\mathcal{U}(\CC_N\fh)$ and
$\mathcal{A}(\fh_{N+1}^\ast)$.

The following result is elementary, and we omit the proof.
\begin{lemt}\label{le:tri}
Let $(-\alpha,d)\in\tilde\Delta^-$, $u\in\mathcal{U}(\CC_N\bar\fn_N^-)$ and let $u'\in \mathcal{U}(\CC_N\fg)$. 
\begin{enumerate}
    \item Assume that $\alpha\in \hat\Delta_{N+1}^+$. Then  
    \[\deg_N \mathbf{q}(u'x_{\alpha}(N-d)u)=-1.\]
    \item Assume that $\mathbf{q}(u')\neq 0$. Then 
        \[\deg_N \mathbf{q}(u'\alpha^\vee(d))\leq \deg_N \mathbf{q}(u')+1\]
        and the equality holds if and only if $\alpha\in\hat\Delta_{N+1}^+$ and $d=N$. 
        In addition, we have 
    \[\deg_N \mathbf{q}(\alpha^\vee(d)u')\leq \max\{\deg_N \mathbf{q}([\alpha^\vee(d),u']),~\deg_N \mathbf{q}(u'\alpha^\vee(d))\}.\]
\end{enumerate}\qed
\end{lemt}
In the rest of this subsection, let
\[P=\{p_{\alpha,d}\}_{(\alpha,d)\in\tilde\Delta^-},\quad R=\{r_{\alpha,d}\}_{(\alpha,d)\in\tilde\Delta^-}\in\mathcal{M}_{N+1},\]
and let $i,j\in I$  such that $(P,i),~(R,j)\in \tilde{I}_\gamma$.  
Define \[P^\ast=\{p^\ast_{\alpha,d}\}_{(\alpha,d)\in\tilde\Delta^-}\in\mathcal{M}_{N+1}\] by letting
\begin{align*}
    p^\ast_{\alpha,d}=p_{\alpha,N-d}.
\end{align*}
By Proposition \ref{det}, the polynomial $\det_\gamma$ is homogeneous with
respect to the grading \eqref{gradpoly}. Hence, $\det_\gamma$ coincides with
its leading homogeneous component. Since the map $P\mapsto P^\ast$ induces a permutation of 
$\tilde I_\gamma$, replacing $P$ by $P^\ast$ in the row index 
changes the determinant only by a sign. Therefore, in order to 
determine $\det_\gamma$ up to a nonzero scalar, it suffices to consider 
the matrix
\begin{align}\label{matrix}
 \left(\mathbf{F}(X_{P^\ast}u_i,X_Ru_j)\right)_{(P,i),(R,j)\in\tilde{I}_\gamma}   
\end{align}
with respect to the grading \eqref{eq:gradh}.

The following result determines the degree of each entry of the matrix \eqref{matrix}.

 \begin{lemt}\label{leql1l2}
One has
\begin{align*}
    \deg_N \mathbf{F}(X_{P^\ast}u_i,X_Ru_j)\leq \mathrm{min}\{|P|,|R|\}.
\end{align*}
\end{lemt}
\begin{proof}
We prove the slightly stronger assertion that
\[
\deg_N\mathbf F(X_{P^\ast}u,X_Rv)
\leq \min\{|P|,|R|\}
\]
for arbitrary $u,v\in\mathcal U(\CC_N\bar{\mathfrak g}_N)$.
The desired statement follows by taking $u=u_i$ and $v=u_j$.
Since the bilinear form $\mathbf{F}$ is symmetric and $|P^\ast|=|P|$, we may assume that $|P|\leq |R|$. The proof proceeds by induction on $|P|$.  If $|P|=0$, then 
\[\mathbf{F}(X_{P^\ast}u,X_Rv)=\mathbf{F}(u,X_Rv)=0\]
unless $|R|=0$, in which case $\deg_N  \mathbf{F}(u,v)\leq 0$. Then the assertion follows.
Now assume that $|P|>0$. Let $x_\alpha(d)$ be the maximal element, with respect to the total order $``\geq "$, among the elements of $\mathcal B$  occurring in $X_{P^\ast}$ with nonzero
multiplicity. Then
\[
X_{P^\ast}
=
x_\alpha(d)X_{(P^\ast)^{(\alpha,d)}}.
\]
By the definition of $\mathbf F$ (see \eqref{deff}), we have
\[
\mathbf F(X_{P^\ast}u,X_Rv)
= \mathbf q\bigl(
\omega(X_{(P^\ast)^{(\alpha,d)}}u)
[x_{-\alpha}(d),X_R]v
\bigr)+
\mathbf q\bigl(
\omega(X_{(P^\ast)^{(\alpha,d)}}u)
X_Rx_{-\alpha}(d)v
\bigr).
\]
By Lemma \ref{le:tri}(1),the second term has degree $-1$. Therefore, it suffices to estimate the first term.

Fix a basis $\mathcal G$ of $\CC_N\fg$ consisting of root vectors and
Cartan elements. By Lemma \ref{commu}, we have
\begin{align*}
[x_{-\alpha}(d),X_R]
={}&
\sum_{\substack{0\leq s\leq N-d\\ r_{\alpha,s}\geq1}}
r_{\alpha,s}X_{R^{(\alpha,s)}}(-\alpha)^\vee(d+s) +
\sum_{\substack{g\in\mathcal G,\ R'\in\mathcal M_{N+1}\\
|R'|\leq |R|-1}}
c_{R',g}X_{R'}g,
\end{align*}
where only finitely many $c_{R',g}\in\C$ are nonzero, and $ c_{R',g}=0 $
whenever $|R'|=|R|-1$ and $g\in\CC_N\fh$.

Applying the induction hypothesis and Lemma \ref{le:tri}, we obtain
\begin{align}
\label{eq:le1}
&\deg_N\mathbf F
\bigl(
X_{(P^\ast)^{(\alpha,d)}}u,
X_{R^{(\alpha,s)}}(-\alpha)^\vee(d+s)v
\bigr)
\leq |P|,\\
\label{eq:le2}
&\deg_N\mathbf F
\bigl(
X_{(P^\ast)^{(\alpha,d)}}u,
X_{R'}gv
\bigr)
\leq
\min\{|P|-1,|R'|+1\},
\end{align}
whenever $g\in\CC_N\fh\cup\CC_N\fn^-$, or
$g=x_\beta(s)$ with
$(-\beta,s)\in\tilde\Delta_{N+1}^-$.

It remains to consider the case
$g\in\CC_N\bar\fn_N^+$.
Note that both
$X_{(P^\ast)^{(\alpha,d)}}$ and $X_{R'}$
are monomials in root vectors $x_{\alpha'}(d')$ with
$\alpha'\in\hat\Delta_{N+1}^-$ and $0\leq d'\leq N$.
For $\alpha'\in\hat\Delta_{N+1}^-$ and
$\beta\in\bar\Delta_N^+$, either
$\pm\alpha'+\beta$ is not a root or
$\pm\alpha'+\beta\in\hat\Delta_{N+1}^{\mp}$.
It follows that
\[
[\omega(X_{(P^\ast)^{(\alpha,d)}})X_{R'},g]
\]
is a $\C$-linear combination of elements
$\omega(X_{P'})X_{R''}$ with
\[
|P'|\leq |P|-1,
\qquad
|R''|\leq |R'|.
\]
By the induction hypothesis,
\begin{align*}
&\deg_N\mathbf F(X_{P'}u,X_{R''}gv)
\leq |P'|
\leq |P|-1,\\
&\deg_N\mathbf F
\bigl(
X_{(P^\ast)^{(\alpha,d)}}\omega(g)u,
X_{R'}v
\bigr)
\leq |P|-1.
\end{align*}
Hence
\begin{align}
\deg_N\mathbf F
\bigl(
X_{(P^\ast)^{(\alpha,d)}}u,
X_{R'}gv
\bigr)
\leq |P|-1.
\label{eq:le3}
\end{align}
This completes the proof.
\end{proof}

For every $u\in\mathcal{U}(\CC_N\fh)$  and every $n\in\BN$, write $u_{(n)}$ for the projection of $u$ onto $\mathcal{U}(\CC_N\fh)_{(n)}$.
 The following result characterizes the leading terms of the diagonal entries in the matrix \eqref{matrix}.
\begin{prpt}\label{pr:l1=l2}
 Assume that $|P|=|R|=m\in\BN$. Then 
\begin{align*}
    \deg_N \mathbf{F}(X_{P^\ast}u_i,X_Ru_j)=m\quad \textrm{if and only if}\quad P=R,~\mathbf{F}(u_i,u_j)\neq 0,
\end{align*}
in which case the leading term  
\begin{align}\label{leadh}
(\mathbf{F}(X_{P^\ast}u_i,X_Ru_j))_{(m)}=\mathbf{F}(u_i,u_j)\prod_{\alpha\in\hat\Delta_{N+1}^-}\prod_{d=0}^{N}p_{\alpha,d}!(-\alpha)^{\vee}(N)^{p_{\alpha,d}}.
\end{align}
\end{prpt}
\begin{proof}
We shall prove the proposition by induction on $m$. There is nothing to prove when $m=0$. Now assume that $m>0$. For convenience, we keep the notation from the proof of Lemma~\ref{leql1l2}. 

Combining \eqref{eq:le1}-\eqref{eq:le3}, we obtain that $ \deg_N \mathbf{F}(X_{P^\ast}u_i,X_Ru_j)=m$ if and only if 
\begin{align}\label{eq:pqm}
    \deg_N \mathbf{F}(X_{(P^\ast)^{(\alpha,d)}}u_i,\sum_{\substack{0\leq s\leq N-d,\\ r_{\alpha,s}\geq 1}}r_{\alpha,s}X_{R^{(\alpha,s)}}(-\alpha)^\vee(d+s)u_j)=m.
\end{align}
By Lemma  \ref{le:tri} (2), this is equivalent to 
\begin{align}\label{equ1}
 r_{\alpha,N-d}\geq 1\quad\textrm{and}\quad    \deg_N \mathbf{F}(X_{(P^\ast)^{(\alpha,d)}}u_i,X_{R^{(\alpha,N-d)}}u_j)=m-1.
\end{align}
Observe that $(P^\ast)^{(\alpha,d)}=(P^{(\alpha,N-d)})^\ast$. It follows from the induction hypothesis that 
\begin{align*}
    \deg_N \mathbf{F}(X_{(P^\ast)^{(\alpha,d)}}u_i,X_{R^{(\alpha,N-d)}}u_j)=m-1 \iff P^{(\alpha,N-d)}= R^{(\alpha,N-d)}, \mathbf{F}(u_i,u_j)\neq 0.
\end{align*}
Then the first assertion follows from \eqref{equ1}.

Assume that $P=R$ and $\mathbf{F}(u_i,u_j)\neq 0$. It follows from \eqref{eq:pqm} that 
\begin{align*}
    (\mathbf{F}(X_{P^\ast}u_i,X_Ru_j))_{(m)}=r_{\alpha,N-d}(-\alpha)^\vee(N)(\mathbf{F}(X_{(P^\ast)^{(\alpha,d)}}u_i,X_{R^{(\alpha,N-d)}}u_j))_{(m-1)},
\end{align*}
while the induction hypothesis yields that 
\begin{align*}
&(\mathbf{F}(X_{(P^\ast)^{(\alpha,d)}}u_i,X_{R^{(\alpha,N-d)}}u_j))_{(m-1)}\\
=\ &\mathbf{F}(u_i,u_j)\prod_{\beta\in\hat\Delta_{N+1}^-}\prod_{k=0}^{N}(p_{\beta,k}-\delta_{\beta,\alpha}\delta_{k,N-d})!(-\beta)^{\vee}(N)^{p_{\beta,k}-\delta_{\beta,\alpha}\delta_{k,N-d}}.
\end{align*}
Then the second assertion follows immediately since $p_{\alpha,N-d}=r_{\alpha,N-d}$.
\end{proof}

The following is the main result of this subsection.
\begin{thmt}\label{th:det}
    Let $\gamma\in Q^+$. Then we have
    \begin{align}\label{eq:det}
        \mathrm{det}_\gamma= k\prod_{(R,i)\in \tilde{I}_\gamma}\left(\prod_{\alpha\in\hat{\Delta}_{N+1}^-}p_{(-\alpha)^{\vee}(N)}^{\sum_{d=0}^N r_{\alpha,d}}\right),
    \end{align}
    where $k$ is a nonzero complex number.
\end{thmt}
\begin{proof}
For $(P,i),(R,j)\in \tilde{I}_\gamma$, let
$f_{P,R}^{i,j}$ be the polynomial on
$\fh_{N+1}^\ast$ defined by
\[
f_{P,R}^{i,j}(\mu)
=
\mathbf B_{\tilde\mu}
(X_{P^\ast}u_{i}v_{\tilde\mu},
X_{R}u_{j}v_{\tilde\mu})\quad \textrm{for all $\mu\in\fh_{N+1}^\ast$}.
\]
By definition,
\[
f_{P,R}^{i,j}(\mu)
=
\tilde\mu\!\left(
\mathbf F\bigl(
X_{P^\ast}u_i,
X_{R}
u_{j}
\bigr)
\right).
\]
In particular, we have $\tilde{\mu}(\mathbf{F}(u_i,u_j))= \mathbf B_{\tilde\mu}(u_i v_{\tilde{\mu}},u_jv_{\tilde{\mu}})=\delta_{i,j}$.
It follows from Lemma \ref{leql1l2} and Proposition \ref{pr:l1=l2} that 
\begin{itemize}
    \item  $\deg_N f_{P,R}^{i,j}\leq \textrm{min}\{|P|,|R|\}$, and 
    \item when $|P|=|R|=m\in \BN$, $\deg_N f_{P,R}^{i,j}=m$ if and only if $(P,i)=(R,j)$.
\end{itemize}
Furthermore, \eqref{leadh} implies that the leading term of $f_{R,R}^{i,i}$ (with respect to the $\BN$-grading \eqref{gradpoly}) is 
\[k_{R,i}\prod_{\alpha\in\hat{\Delta}_{N+1}^-}p_{(-\alpha)^{\vee}(N)}^{\sum_{d=0}^N r_{\alpha,d}},\]
where $k_{R,i}\in\C^{\times}$. Therefore,  the leading term of $\det_\gamma$ coincides with the right-hand side of \eqref{eq:det}.

On the other hand, since each $p_{\alpha^\vee(N)}\in\mathcal{A}(\fh_{N+1}^\ast)_{(1)}$, it follows from Proposition \ref{det}  that
 $\det_\gamma$ is a homogeneous polynomial with respect to the $\BN$-grading of $\mathcal{A}(\fh_{N+1}^\ast)$ (see \eqref{gradpoly}).
Hence $\det_\gamma$ coincides with its leading term, which is precisely the polynomial in
\eqref{eq:det}.
\end{proof}

\subsection{Proof of Theorem \ref{thm:main}} For $N=0$, Theorem \ref{thm:main} reduces to Jantzen's criterion. Hence we assume $N\ge1$. Theorem \ref{thm:main} follows from Theorem  \ref{th:mainirr} and the following result.
\begin{prpt}
    Assume that $\chi$ is not nonsingular. Then $M_{\mathcal{P}_{\bm\Pi}}(\chi)$ has a proper submodule.
\end{prpt}
\begin{proof}
We begin with the case where $\chi_1$ does not satisfy condition (J). 
By  Lemma \ref{pullback}, Proposition \ref{giiso} and Proposition \ref{pr:0}, the $\CC_N\bar{\mathfrak{g}}_1$-module $W_\chi(\bar{\mathfrak{g}}_1)=\mathcal{U}(\CC_N\bar{\mathfrak{g}}_1).v_{\chi}$ is reducible. Choose a nonzero proper submodule $S$ and a nonzero weight vector
\[
w=uv_\chi \in S, \qquad u\in\mathcal U(\CC_N\bar{\mathfrak n}_1^-),
\]
whose weight is maximal among the weights of $S$.
That is, $\mathrm{wt\,}w'-\mathrm{wt\,}w\notin Q^+$ for all  nonzero weight vectors $w'\in S$ with $\mathrm{wt\,}w'\neq \mathrm{wt\,}w$. Then 
\begin{align}\label{sinpi1}
x_\alpha(k)w=0
\qquad
\text{for all }\alpha\in\Pi_0\cup\Pi_1,\quad
0\leq k\leq N.
\end{align} 
For $\alpha\in  \cup_{j=2}^{N+1}\Pi_j$ and  $\beta\in\bar\Delta_1^+$, we have 
 $\beta-\alpha\notin \Delta\cup\{0\}$. It follows that 
\[[x_{\alpha}(k),u]=0\]
and hence 
\begin{align}\label{eq:singu}
     x_{\alpha}(k).w=0\quad \textrm{for }\alpha\in  \cup_{j=2}^{N+1}\Pi_j,~0\leq k\leq N.
\end{align}
Combining \eqref{sinpi1} and \eqref{eq:singu},
we see that $w$ is annihilated by all generators of
$\CC_N\fn^+$.
Hence, the $\CC_N\fg$-submodule generated by $w$
is a proper submodule of
$M_{\mathcal P_{\bm\Pi}}(\chi)$ by the Poincaré-Birkhoff-Witt theorem.

Now assume that $\chi_1$ satisfies the condition (J). Since $\chi$ is not nonsingular, choose $i\in\{2,3,\dots,N+1\}$  minimal  such that 
\begin{align*}
    \chi(\beta^{\vee}(j-1))\neq 0\quad \textrm{for all }\beta\in\hat\Delta_j^{+},~j=2,\dots,i-1,
\end{align*}
and
\begin{align}\label{sinalpha}
      \chi(\alpha^{\vee}(i-1))=0\quad \textrm{for some }\alpha\in\hat\Delta_i^{+}.
\end{align}
Thus, $\chi_{i-1}\in \fh_{\bm\Pi^{i-1}}^\ast$ is nonsingular, whereas $\chi_i\in\fh_{\bm\Pi^{i}}^\ast$ is not nonsingular.

 Consider the parabolically induced
$\CC_{i-1}\bar{\mathfrak g}_i$-module
$M_{\bm\Pi^i}(\chi_i)$. The corresponding Shapovalov form is defined as in Subsection \ref{shap}, with  
$(N,\bm\Pi)$ replaced by $(i-1,\bm\Pi^i)$.
It follows from Theorem~\ref{th:det} and \eqref{sinalpha} that
\[
{\det}_{\alpha}(\chi_i)=0.
\]
This implies that  the $\CC_{i-1}\bar{\mathfrak g}_i$-module
$M_{\bm\Pi^i}(\chi_i)$ is reducible by Lemma \ref{rad}, and hence the $\CC_N\bar\fg_{i}$-module $W_\chi(\bar\fg_{i})$ is reducible as well by Lemma \ref{pullback} and Proposition \ref{giiso}.
Therefore, there exists a nonzero singular vector $w'\in W_\chi(\bar\fg_{i})\backslash \C v_\chi$ such that $x_{\alpha}(k).w'=0$ for all $\alpha\in \overline{\Pi}_i$ and $0\leq k\leq N$. On the other hand,  we obtain that
\[x_{\alpha}(k).w'=0\quad \textrm{for }\alpha\in\cup_{j=i+1}^{N+1}\Pi_j,~0\leq k\leq N\]
by an argument  similar to that for \eqref{eq:singu}.
Consequently, the
$\CC_N\fg$-submodule generated by $w'$
is a proper submodule of
$M_{\mathcal P_{\bm\Pi}}(\chi)$, as desired.
\end{proof}


\section{Irreducible highest weight modules for loop algebras}\label{sec:6}
In this section, we study irreducible highest weight modules for loop  algebras and obtain character formulas for a class of such modules.
\subsection{Highest weight modules for loop algebras}\label{sec:6.1}
Let $\mathcal{L}=\C[t,t^{-1}]$ be the Laurent polynomial ring and let $\varphi: \CL\mathfrak{h}\rightarrow \BC$ be a linear functional on $\CL\h$. 
Extend it to be a character of $\CL\b$,  still denote by $\varphi$, by requiring that $\varphi(\CL\n^+)=0$. Form 
the Verma $\CL\g$-module
\[
M(\varphi)=\mathrm{Ind}_{\CL\b}^{\CL\g} ~\C_\varphi.
\]
Write $V(\varphi)$ for the irreducible quotient of $M(\varphi)$. 

We say that $\varphi$ is polynomial if for every $\alpha\in \Pi$, there is a polynomial $f_\alpha(t)$ such that 
\begin{align}\label{poly}
    \varphi(\alpha^\vee\otimes t^n)=f_\alpha(n)\quad \textrm{ for all $n\in \Z$}.
\end{align}
 More generally, 
we say that $\varphi$ is exp-polynomial if there exist finitely many nonzero complex numbers $a_1, a_2,\dots, a_\nu$ and polynomial characters $\varphi_1,\varphi_2,\dots,\varphi_\nu$ such that 
\be\label{eq:exp-poly}
\varphi(h\otimes t^n)=\sum_{i=1}^\nu a_i^n \varphi_i(h\otimes t^n)\quad \te{for all }h\in \h,~n\in\Z.
\ee

Write $\mathcal{E}$ for the subspace of $(\CL\fh)^\ast$ consisting of all exp-polynomial linear functionals on $\CL\fh$.

The following result is obtained in \cite[Theorem 2.9]{BGLZ}.
\begin{prpt}\label{pr:quasifinite}
The irreducible $\CL\g$-module $V(\varphi)$ has finite-dimensional weight spaces if and only if $\varphi\in\mathcal{E}$.
\end{prpt}

Let $V$ be an $\CL\g$-module and $a\in \C^*$. We denote by $V[a]$ the $\CL\g$-module whose underlying space is $V$ and the module action is given by 
\[
(x\otimes t^n). v= a^n (x\otimes t^n).v\quad \te{for }x\in\g,~n\in\Z,~\te{and }v\in V.
\]
Let $\varphi$ be an exp-polynomial character on $\CL\h$ as in \eqref{eq:exp-poly}. When $\varphi\neq0$, we may and do assume that $a_1,\dots,a_\nu$ are pairwise distinct and that
$\varphi_1,\dots,\varphi_\nu$ are nonzero. Then the decomposition \eqref{eq:exp-poly} is unique up to a permutation of the indices. In this case, we have the following result (see \cite[Theorem 1.7]{CP}). 

\begin{prpt}\label{tensor}
The tensor product $\CL\g$-module 
\[V(\varphi_1)[a_1]\otimes V(\varphi_2)[a_2]\otimes \cdots \otimes V(\varphi_\nu)[a_\nu]\] 
is irreducible. In particular,
  \[
  V(\varphi)\cong V(\varphi_1)[a_1]\otimes V(\varphi_2)[a_2]\otimes \cdots \otimes V(\varphi_\nu)[a_\nu].
  \]
\end{prpt}

\subsection{Connection with Takiff algebras}\label{sec:6.2}
In this subsection, we assume that $\varphi$ is a nonzero polynomial character on $\CL\h$. 
For each $\alpha\in \Pi$, write $f_\alpha(t)\in \C[t]$ such that 
\begin{align}\label{eq:falpha}
    \varphi(\alpha^\vee\otimes t^n)=f_\alpha(n)\quad \textrm{for all }n\in\Z.
\end{align}
Set
\[ N=\max\{\deg f_\alpha(t)\mid \alpha\in \Pi\},\]
which is a nonnegative integer. 
Here, as a convention, we assume that $\deg f_\alpha(t)=-1$ if $f_\alpha(t)=0$.

Note that the algebra automorphism
\[ \C[t]\rightarrow \C[t],\quad t\mapsto t-1 \] induces an isomorphism
\[
i_N:\CC_N\to \C[t]/(t-1)^{N+1}\C[t].
\]
On the other hand, the natural embedding $\C[t]\hookrightarrow \CL$
 induces an isomorphism
\[
j_N:\C[t]/(t-1)^{N+1}\C[t]
\to
\CL/(t-1)^{N+1}\CL.
\]
Then we have the composition map 
\[
\pi_N:\ \CL\xrightarrow{\text{the quotient map}}
 \CL/(t-1)^{N+1}\CL\xrightarrow{j_N^{-1}}\C[t]/(t-1)^{N+1}\C[t]\xrightarrow{i_N^{-1}}\CC_N.
\]
For every subalgebra $\mathfrak{a}$ of $\mathfrak{g}$, $\pi_N$ induces a surjective   homomorphism 
\[
\pi_{N,\mathfrak{a}}:\ \CL\mathfrak{a}\rightarrow \CC_N\mathfrak{a},\quad x\otimes f(t)\mapsto x\otimes \pi_N(f(t)).
\]
By definition, the kernel of $\pi_{N,\mathfrak{a}}$ is $\mathfrak{a}\otimes (t-1)^{N+1}\CL$.


\begin{lemt}
One has that
\[\varphi(\h\otimes (t-1)^{N+1}\CL)=0.\]
\end{lemt}
\begin{proof}
It suffices to prove that
\[
\varphi(\alpha^\vee\otimes (t-1)^{N+1}t^m)=0
\]
for all $\alpha\in\Pi$ and $m\in\Z$. Since
$\deg f_\alpha(t)\leq N$, applying \cite[Eq.~(3.7)]{CGT} gives
\[
\sum_{k=0}^{N+1}
\binom{N+1}{k}(-1)^{N+1-k}f_\alpha(m+k)=0.
\]
Hence, by \eqref{eq:falpha},
\begin{align*}
\varphi(\alpha^\vee\otimes (t-1)^{N+1} t^m)
={}&
\sum_{k=0}^{N+1}
\binom{N+1}{k}(-1)^{N+1-k}
\varphi(\alpha^\vee\otimes t^{m+k})\\
={}&
\sum_{k=0}^{N+1}
\binom{N+1}{k}(-1)^{N+1-k}
f_\alpha(m+k)\\
={}&0,
\end{align*}
as desired.
\end{proof}
The above lemma implies that via the homomorphism $\pi_{N,\h}$, $\varphi$ induces a character on $\CC_N\h$, say $\bar\varphi$.  

\begin{prpt}\label{pr:pbpichi}
    One has the following $\CL\g$-module isomorphism:
    \[
    V(\varphi)\cong \pi_{N,\g}^{-1}(V_{\CC_N\mathfrak{b}}(\bar\varphi)).
    \]
\end{prpt}
\begin{proof}
    The assertion follows from Lemma \ref{pullback}.
\end{proof}
In view of Propositions \ref{pr:quasifinite}, \ref{tensor}, and \ref{pr:pbpichi}, the character formula problem for irreducible highest weight $\CL\fg$-modules with finite-dimensional weight spaces reduces to that for irreducible modules of Takiff algebras associated with $\fg$.

\subsection{Character formula of $V(\varphi)$}\label{sec:charloop}
In this subsection, let $\varphi\in\mathcal{E}$. Then the $\CL\fg$-module $V(\varphi)$ has finite-dimensional weight spaces by Proposition \ref{pr:quasifinite}.
 We will present an explicit character formula for $V(\varphi)$. 

 When $\varphi=0$, it is obvious that $\mathrm{char\,} V(\varphi)=e(0)$. We henceforth assume that $\varphi\neq 0$. Then  there exist finitely many  distinct complex numbers $a_1, a_2,\dots, a_\nu$ and nonzero polynomial characters $\varphi_1,\varphi_2,\dots,\varphi_\nu$ such that $\varphi$ satisfies \eqref{eq:exp-poly}.

 Let $1\leq i\leq \nu$. For $\alpha\in \Pi$, write $f_{\alpha,i}(t)\in\C[t]$ such that 
\[\varphi_i(\alpha^\vee\otimes t^n)=f_{\alpha,i}(n)\quad \textrm{for all }n\in\Z.\]
Set
\[
N(\varphi_i) = \max\{\deg f_{\alpha,i}(t)\mid \alpha\in\Pi\},
\]
where we set $\deg f_{\alpha,i}(t)=-1$ if
$f_{\alpha,i}(t)=0$. For simplicity, write
$N_i=N(\varphi_i)$. As in the preceding subsection,
$\varphi_i$ induces a character $\bar\varphi_i$ on
$\CC_{N_i}\fh$. Recall that, for every $\alpha\in\Pi$,
\[
(\bar\varphi_i)_\alpha = \begin{cases} \max\{1\leq j+1\leq N_i+1 \mid \bar\varphi_i(\alpha^\vee(j))\neq0\},
&\mbox{if } f_{\alpha,i}(t)\neq0,\\
0,&\mbox{if } f_{\alpha,i}(t)=0.
\end{cases}
\]
By \cite[Eq.~(3.7)]{CGT}, one easily checks that
\[
(\bar\varphi_i)_\alpha= \deg f_{\alpha,i}(t)+1
\]
for all $\alpha\in\Pi$.
Thus we obtain a genuine $(N_i+2)$-partition
\[
\bm\Pi(\bar\varphi_i) = \left(
\Pi(\bar\varphi_i)_0, \Pi(\bar\varphi_i)_1, \dots, \Pi(\bar\varphi_i)_{N_i+1} \right)
\]
of $\Pi$, where
\[
\Pi(\bar\varphi_i)_0 = \{\alpha\in\Pi\mid f_{\alpha,i}(t)=0\},
\]
and
\[
\Pi(\bar\varphi_i)_j = \{\alpha\in\Pi\mid
\deg f_{\alpha,i}(t)=j-1\},
\qquad 1\leq j\leq N_i+1.
\]
Associated with this partition, we  have 
\[
\fh_{\bm\Pi(\bar\varphi_i)}
=
\sum_{k=1}^{N_i+1}
\sum_{\alpha\in\Pi(\bar\varphi_i)_k}
\sum_{j=0}^{k-1}
\C\alpha^\vee(j)
\subseteq
\CC_{N_i}\fh.
\]

\begin{dfnt}
 Let $\varphi\in(\CL\fh)^\ast$ be as in \eqref{eq:exp-poly}. We say that $\varphi$ is nonsingular if the induced  character $\bar{\varphi}_i|_{\fh_{\bm\Pi(\bar{\varphi}_i)}}$  is nonsingular for all $i=1,2,\dots,\nu$.
\end{dfnt}

Write $\bar\Delta_j^i(\varphi)$ for the root subsystem of $\Delta$ generated by $\cup_{k=0}^j\Pi(\bar{\varphi}_i)_k$  for $0\leq j\leq N_i+1$ and write
\[(\hat\Delta_0^{i,-}(\varphi),\hat\Delta_1^{i,-}(\varphi),\dots,\hat\Delta_{N_i+1}^{i,-}(\varphi))\]
for the associated  $(N_i+2)$-partition 
of $\Delta^-$, where $\hat\Delta_0^{i,-}(\varphi)=\bar\Delta_0^i(\varphi)\cap \Delta^-$ and
\[\hat\Delta_j^{i,-}(\varphi)=\left(\bar\Delta_j^i(\varphi)\backslash\bar\Delta_{j-1}^i(\varphi)\right)\cap\Delta^-\quad \textrm{for }j=1,2,\dots, N_i+1.\]

For $a\in\C^\ast$, one has $V(\varphi_i)[a]\cong V(\varphi_i)[1]$ as $\fh$-modules.
Then it follows from Propositions \ref{tensor}, \ref{pr:pbpichi} that 
\begin{align}\label{char:loop-tak}
    \mathrm{char}\, V(\varphi)=\prod_{i=1}^{\nu}\mathrm{char}\, V_{\CC_{N_i}\mathfrak{b}}(\bar\varphi_i),
\end{align}
while the character formula for $V_{\CC_{N_i}\mathfrak{b}}(\bar\varphi_i)$ is obtained in Theorem \ref{th:chartak} whenever $\bar{\varphi}_i|_{\fh_{\bm\Pi(\bar{\varphi}_i)}}$  is nonsingular. 

Thus, we obtain the following result.

\begin{thmt}\label{th:char-loop}
Let $0\neq\varphi\in\mathcal{E}$ be as in \eqref{eq:exp-poly}. Assume that $\varphi$ is nonsingular. Then 
\begin{align*}
    \mathrm{char}\, V(\varphi)= e(\varphi|_{\fh})\prod_{i=1}^{\nu}\left(\prod_{j=1}^{N_i+1}\prod_{\alpha\in\hat\Delta_j^{i,-}(\varphi)}(1-e(\alpha))^{-j}\right).
\end{align*}
\end{thmt}

\section{Character formulas for the category $\widetilde{\mathcal{O}}$}\label{sec:7}
In this section, we present the character formulas for a class of irreducible Harish-Chandra modules in the category  $\widetilde{\mathcal{O}}$ for the affine Lie algebra associated with $\fg$ by applying Chari-Pressley’s loop module construction (see \cite{CP}).
\subsection{Chari's category $\widetilde{\mathcal{O}}$}\label{sec:7.1}
 Let
 \[\wt\fg=\CL\fg\oplus \C\rd\] denote the centerless affine Lie algebra associated with $\fg$, with the remaining Lie bracket given by
\begin{align*}
[\rd, x\ot t^m]=m\, x\ot t^m\quad\textrm{for }x\in\fg,~m\in\Z.
\end{align*}
Set 
\[
\text{$\wt\fh=  \fh\oplus\C\rd$\quad and\quad  $\wh\fh=\CL\fh\oplus\C\rd$.}\]
Let $\delta\in\wt\fh^*$ be defined by  $\delta|_{\fh}=0$ and $\delta(\rd)=1$. As usual, a $\wt\fg$-module (resp.\,$\wh\fh$-module) $V$ is called  a weight $\wt\fg$ (resp.\,$\wh\fh$)-module if it admits a weight space decomposition with respect to $\wt\fh$. Namely, 
 $V=\oplus_{\beta\in \wt\fh^{\ast}}V_\beta$,
 where
 \begin{align*}
     V_\beta=\{v\in V\mid h.v=\beta(h)v,\ \forall\,h\in \wt\fh\}.
 \end{align*}
We write
\[\mathcal{P}(V)=\{\beta\in \wt\fh^{\ast}\mid V_\beta\neq 0 \}\]
for the set of weights of  $V$. A weight $\wt\fg$-module $V$ is called a Harish-Chandra module if  \[\dim\, V_\beta<\infty\quad \textrm{for all }\beta\in \mathcal{P}(V).\]
Throughout this section, we consider Harish-Chandra $\wt\fg$-modules in the category $\widetilde{\mathcal{O}}$.

The category $\widetilde{\mathcal{O}}$ of $\wt\fg$-modules introduced by Chari \cite{C} is defined as follows: a module $V$ belongs to $\widetilde{\mathcal O}$ if and only if $V$ is a weight $\wt\fg$-module and there exist finitely many elements $\lambda_1,\lambda_2,\dots,\lambda_r\in \wt\fh^\ast$ such that 
\[\mathcal{P}(V)\subseteq \bigcup_{i=1}^r\{\lambda_i-\alpha+n\delta\mid \alpha\in Q^+,~n\in\mathbb{Z}\}.\]
The morphisms in the category $\widetilde{\mathcal{O}}$ are  $\wt\fg$-module homomorphisms.

Let $\varphi\in (\CL\fh)^\ast$ and let $b\in\mathbb{C}$. Define an $\wh\fh$-module structure on
 $\CL$ by setting
\begin{align*}
(h\ot t^m).t^n=\varphi(h\ot t^m) t^{m+n}\quad\te{and}\quad \rd.t^n=(n+b)\, t^n
\end{align*}
for all $h\in\fh,~m,n\in\Z$.
 Denote by $\CL(\varphi,b)$ the resulting $\wh\fh$-module, and write $\CL_{\varphi,b}$ for
the $\wh\fh$-submodule of $\CL(\varphi,b)$ generated by $1$. Let $A_\varphi$ be the  submonoid of $\BZ$ generated by the set
$
\{m\in\BZ\mid \varphi(h\ot t^m)\neq 0\textrm{ for some }h\in\fh\}.
$
Then \[\CL_{\varphi,b}=\textrm{Span}_\C\{t^m\mid m\in A_\varphi\}.\] Set
\[\mathcal{F}^\ast=\{\varphi\in (\CL\fh)^\ast\mid A_\varphi=r\Z\textrm{ for some }r\in\BN\}.\]

The following result is straightforward (cf. \cite{C}).
\begin{lemt}
    The $\wh\fh$-module $\CL_{\varphi,b}$ is irreducible if and only if $\varphi\in \mathcal{F}^\ast$.
\end{lemt}

Let $\varphi\in\mathcal F^\ast$. Regard  $\CL_{\varphi,b}$ as an $\wh\fh\oplus \CL\fn^+$-module by letting $\CL\fn^+$ act trivially. Form the induced $\wt\fg$-module 
\begin{align*}
    \widetilde{M}(\varphi,b):=\mathcal{U}(\wt\fg)\otimes_{\mathcal{U}(\wh\fh\oplus \CL\fn^+)}\CL_{\varphi,b}.
\end{align*}
Write $\widetilde{V}(\varphi,b)$ for the irreducible quotient of $\widetilde{M}(\varphi,b)$. It is straightforward to verify that $\widetilde{V}(\varphi,b)$ is an irreducible object in $\widetilde{\mathcal{O}}$.
 Conversely, it was proved in \cite{C} that  every irreducible object of $ \widetilde{\mathcal{O}}$ is isomorphic to $\wt{V}(\varphi,b)$ for some $\varphi\in \mathcal{F}^\ast$ and $b\in\C$. 

We have the following result (cf. \cite{BGLZ, CP}).
\begin{prpt}\label{fexp}
Let $\varphi\in(\CL\fh)^\ast$ and let $b\in\C$.
\begin{enumerate}
    \item  If $\varphi$ is  exp-polynomial, then $\varphi\in \mathcal{F}^\ast$.
    \item   If $\varphi\in \mathcal{F}^\ast$, then $\widetilde V(\varphi,b)$ is a Harish-Chandra $\wt\fg$-module if and only if $\varphi$ is exp-polynomial.  
\end{enumerate}
\end{prpt}

\subsection{Character formula for $\wt{V}(\varphi,b)$}
Let $V$ be a Harish-Chandra $\wt\fg$-module. As usual, write
\begin{equation*}
\textrm{char}\,V\ =\ \sum_{\lambda\in\wt\fh^{\ast}}\textrm{dim\,}V_{\lambda}\cdot e(\lambda)
\end{equation*}
for the character of $V$. 

Let $(\varphi,b)\in\mathcal F^\ast\times\C$, and assume that $\widetilde V(\varphi,b)$ is a Harish-Chandra module. By Proposition \ref{fexp}, we have $\varphi\in\mathcal E$. If $\varphi=0$, then $\widetilde V(\varphi,b)$ is the one-dimensional $\wt\fg$-module on which $\CL\fg$ acts trivially and $\rd$ acts as the scalar $b$. Hence
\[
\operatorname{char}\widetilde V(\varphi,b)=e(b\delta).
\]
We henceforth assume that $\varphi\neq0$. Since
$\varphi\in\mathcal F^\ast$, we have
\[
A_\varphi=r\Z
\]
for some $r\in\Z_+$. Then $\wt{V}(\varphi,b)$ has finite-dimensional weight spaces by Proposition \ref{fexp}. Write $\varepsilon$ for a primitive $r$-th root of unity. By \cite{CP}, we may (and do) write 
\begin{align}\label{eq:decomvarphi}
\varphi=\varphi^0+\varphi^1+\dots+\varphi^{r-1}, 
\end{align}
where  $\varphi^j\in\mathcal{E}$ for $0\leq j\leq r-1$, satisfies
\begin{align*}
\varphi^j(\alpha^\vee\otimes t^n)
=\sum_{i=1}^{\nu}(\varepsilon^j a_i)^n f_{\alpha,i}(n)
\qquad \te{for }\alpha\in\Pi,~n\in\Z.
\end{align*}
Here, $a_1,\dots,a_\nu$ are distinct nonzero complex numbers with $
  (a_p/a_q)^r\neq 1$ for $1\leq p\neq q\leq \nu$, and
$f_{\alpha,i}(t)\in\C[t]$ for all $\alpha\in\Pi$ and $1\leq i\leq\nu$. For later use, denote by $\varphi_i$  the polynomial character of $\CL\fh$ determined by the polynomials $f_{\alpha,i}(t)$, $\alpha\in\Pi$. That is,
\begin{align*}
    \varphi_{i}(\alpha^\vee\otimes t^n)=f_{\alpha,i}(n)\quad\te{for all }\alpha\in\Pi,~n\in\BZ. 
\end{align*}


By Proposition \ref{tensor}, for $0\leq j\leq r-1$, we have
\[
V(\varphi^j)  \cong V(\varphi_1)[\varepsilon^ja_1]\otimes\cdots\otimes
V(\varphi_\nu)[\varepsilon^ja_\nu].
\]
Since twisting does not change the underlying vector space, we identify the $V(\varphi^j)$, $0\leq j\leq r-1$, with the same vector space. Moreover, Proposition \ref{tensor} gives
\[
V(\varphi) \cong V(\varphi^0)\otimes\cdots\otimes V(\varphi^{r-1})
\]
as vector spaces.
Consider the cyclic permutation $\sigma=(r-1,r-2,\dots,0)$. It induces a linear automorphism of $V(\varphi)$ of order $r$, 
still denoted by $\sigma$, given by
\[\sigma(v_0\otimes v_1\otimes \cdots\otimes v_{r-1})=v_{r-1}\otimes v_0\otimes\cdots\otimes v_{r-2}.\]
Then $\sigma$ is diagonalizable and hence
\[V(\varphi)=\bigoplus_{j=0}^{r-1}V(\varphi)_j,\]
where 
\[V(\varphi)_j=\{v\in V(\varphi)\mid \sigma(v)=\varepsilon^jv\}\]
denotes the eigenspace corresponding to the eigenvalue $\varepsilon^j$. 

For every $n\in\Z$, write $\bar{n}$ for the unique element in $\{0,1,\dots,r-1\}$ such that $n-\bar{n}\in r\Z$.
 \begin{lemt}\label{le:submod}
     Let $j=0,1,\dots,r-1$, $x\in \fg$, and let $m\in\Z$. Then
     \[(x\otimes t^m).V(\varphi)_j\subseteq V(\varphi)_{\overline{j-m}}.\]
 \end{lemt}
 \begin{proof}
    By Proposition \ref{tensor}, we obtain
     \begin{align}
         V(\varphi^j)=V(\varphi_1)[\varepsilon^ja_1]\otimes V(\varphi_2)[\varepsilon^ja_2]\otimes\cdots\otimes V(\varphi_\nu)[\varepsilon^ja_\nu]
     \end{align}
     as $\CL\fg$-modules.
     For the $\CL\fg$-module $V(\varphi^j)$, write $\phi_{\varphi^j}:\CL\fg\to \textrm{End} (V(\varphi^j))$ for the corresponding Lie algebra homomorphism. 
 Then we have that 
 \begin{align*}
     &\phi_{\varphi^j}(x\otimes t^m)(u_{1}\otimes u_2\otimes\cdots\otimes u_{\nu})\\
     =\ &\varepsilon^{jm}\sum_{i=1}^{\nu}a_i^mu_{1}\otimes\cdots \otimes u_{i-1}\otimes (x\otimes t^m).u_i\otimes u_{i+1}\otimes\cdots\otimes u_\nu\\
     =\ &\varepsilon^{-m}\phi_{\varphi^{\overline{j+1}}}(x\otimes t^m)(u_{1}\otimes u_2\otimes\cdots\otimes u_{\nu}),
 \end{align*}
 where $u_i\in V(\varphi_i)[\varepsilon^ja_i]$ and $1\leq i\leq \nu$.
 This implies 
 \begin{align*}
\sigma((x\otimes t^m).v)=\varepsilon^{-m}(x\otimes t^m).\sigma(v)=\varepsilon^{j-m}(x\otimes t^m).v\quad \textrm{for all }v\in V(\varphi)_j,
 \end{align*}
as desired.
 \end{proof}
Recall the loop module construction from \cite{CP}. Assume that $W$ is a highest weight $\CL\fg$-module with highest weight $\varphi$. Then there exists a natural $\wt\fg$-module structure on the loop space $W\ot\CL$ such that the action is given by 
\[(x\ot t^n).(w\ot t^m)\ =\ (x\ot t^n.w)\ot t^{m+n}\quad \te{and}\quad\rd.(w\ot t^m)\ =\ (m+b)w\ot t^m,\]
where $x\in\fg$, $m,n\in\Z$ and $w\in W$. Denote by  $\CL(W,b)$ the resulting $\wt\fg$-module. 

Extend $\sigma$ to a linear automorphism $\widetilde{\sigma}$ of $\CL(V(\varphi),b)$ by letting
\begin{align}\label{eqsigma}
    \widetilde{\sigma}(v\otimes t^n)=\varepsilon^n \sigma(v)\otimes t^n\quad \textrm{for }v\in V(\varphi),~ n\in \Z.
\end{align}
We have the following eigenspace decomposition with respect to $\widetilde{\sigma}$:
\begin{align*}
    \CL(V(\varphi),b)=\bigoplus_{j=0}^{r-1} \CL(V(\varphi),b)_{j},
\end{align*}
where 
\[\CL(V(\varphi),b)_j=\{v\in \CL(V(\varphi),b)\mid \widetilde{\sigma}(v)=\varepsilon^j v\}.\] 
Then it follows from \eqref{eqsigma} that
\begin{align}\label{eq:lchibj}
    \CL(V(\varphi),b)_j=\bigoplus_{n\in\Z}V(\varphi)_{\overline{j-n}}\otimes \C t^n\quad \textrm{for }j=0,1,\dots,r-1.
\end{align}
Furthermore, Lemma \ref{le:submod} implies that 
\begin{align*}
    (x\otimes t^m).(V(\varphi)_{\overline{j-n}}\otimes \C t^n)\subseteq V(\varphi)_{\overline {j-m-n}}\otimes \C t^{m+n} \subseteq \CL(V(\varphi),b)_j
\end{align*}
for all $x\in\fg$ and $m\in\BZ$.
Consequently, each eigenspace $\CL(V(\varphi),b)_j$ is an $\wt\fg$-submodule of $\CL(V(\varphi),b)$.

The following result is obtained in \cite[Theorem 4.7]{CP}.
\begin{prpt} \label{vchibi}
Let $i=0,1,\dots,r-1$. There is an isomorphism of $\wt\fg$-modules
\[\CL(V(\varphi),b)_i\cong \wt{V}(\varphi,b+i).\]
\end{prpt}

Now we can compute the character formula for $\wt{V}(\varphi,b)$ using the eigenspace decomposition above. To state the character formula, we use the following standard arithmetic functions.
Let 
\[\phi:\Z_{+}\to \Z_{+}\]
denote  Euler's totient function such that 
\[\phi(d)=\operatorname{Card}\{n\in\BN\mid 1\leq n\leq d,~\gcd (d,n)=1\}\quad\te{for }d\in\Z_{+}.\]
Let 
\[\mu:\Z_+\to\{-1,0,1\}\]
denote the Möbius function, defined by $\mu(1)=1$,
$\mu(d)=(-1)^\ell$ if $d$ is a product of $\ell$  distinct primes, and $\mu(d)=0$ otherwise. For every positive integer $d$ and integer $n$, write
\[
c_d(n)=\frac{\phi(d)\mu(d')}{\phi(d')}
\quad \big(d'=\frac{d}{\gcd(d,n)} \big)
\]
for the Ramanujan sum. 

\begin{prpt}  \label{pr:char-affine}
Let $\varphi\in\mathcal E$ satisfy
$A_\varphi=r\Z$ for some $r\in\Z_+$, and let
\eqref{eq:decomvarphi} be the decomposition above. Let $b\in\C$.  Then 
\begin{align*}
\mathrm{char}\,\wt{V}(\varphi,b)\ =\ \frac{1}{r}\sum_{n\in\Z}\sum_{d\mid r}c_{d}(\bar{n})(\mathrm{char}\,V(\varphi^0,d))^{\frac{r}{d}}e((-n+b)\delta),
\end{align*}
where $\mathrm{char}\,V(\varphi^0,d)$ is obtained from $\mathrm{char}\,V(\varphi^0)$ by replacing $e(\alpha)$ by $e(d\alpha)$ for every $\alpha\in\fh^\ast$.
\end{prpt}
\begin{proof}
Note that for every $n\in\Z$, the subspace $V(\varphi)_{\overline{n}}$ is $\fh$-invariant.
By \eqref{eq:lchibj} and Proposition \ref{vchibi}, we have 
\begin{align}
    \mathrm{char}\,\wt{V}(\varphi,b)= \sum_{n\in\Z} \mathrm{char}\,(V(\varphi)_{\overline{n}})e((-n+b)\delta).
\end{align}
 As $\fh$-modules, we have 
\[V(\varphi)\cong V(\varphi^0)^{\otimes r}.\]
Then by \cite[Proposition B.2]{HLN} (cf. \cite[Theorem 5.5]{W}),
 we obtain
\begin{align*}
    \dim\,\left((V(\varphi^0)^{\otimes r})_{\overline{n}}\right)_{\varphi|_\fh-\alpha}= \frac{1}{r}\sum_{d\mid r}c_{d}(\bar{n})\dim\,\left((V(\varphi^0)^{\otimes \frac{r}{d}}\right)_{\frac{1}{d}\varphi|_\fh-\frac{\alpha}{d}}
\end{align*}
for every $\alpha\in Q^+$. It follows that 
\begin{align*}
   \mathrm{char}\,V(\varphi)_{\overline{n}}=\ & \sum_{\alpha\in Q^+}\dim\, \left((V(\varphi^0)^{\otimes r})_{\overline{n}}\right)_{\varphi|_\fh-\alpha} e(\varphi|_\fh-\alpha)\\
   =\  &\frac{1}{r}\sum_{\alpha\in Q^+}\sum_{d\mid r}c_{d}(\bar{n})\dim\, \left((V(\varphi^0)^{\otimes \frac{r}{d}}\right)_{\frac{1}{d}\varphi|_\fh-\frac{\alpha}{d}} e(\varphi|_\fh-\alpha)\\
   =\ &\frac{1}{r}\sum_{d\mid r}c_{d}(\bar{n})\left(\mathrm{char}\,V(\varphi^0,d)\right)^{\frac{r}{d}},
\end{align*}
where $\mathrm{char}\,V(\varphi^0,d)$ is obtained from $\mathrm{char}\,V(\varphi^0)$ by replacing $e(\alpha)$ by $e(d\alpha)$ for every $\alpha\in\fh^\ast$. 
\end{proof}

Recall the notation $N(\varphi_{i})$ and $\hat\Delta_j^{i,-}(\varphi^0)$ introduced in Subsection \ref{sec:charloop}.
Note that $\varphi^0$ is nonsingular if and only if  $\varphi$ is nonsingular. Combining Theorem \ref{th:char-loop} and Proposition \ref{pr:char-affine}, we obtain an explicit character formula for $\wt{V}(\varphi,b)$ in the nonsingular case. 
\begin{thmt}\label{th:cha}
Let $\varphi\in\mathcal E$ satisfy
$A_\varphi=r\Z$ for some $r\in\Z_+$, and let
\eqref{eq:decomvarphi} be the decomposition above. Let $b\in\C$. Assume that $\varphi$ is nonsingular. Then 
\begin{align*}
&\mathrm{char}\,\wt{V}(\varphi,b) \\
=\ &\ \frac{1}{r}e(\varphi|_\fh)\sum_{n\in\Z}\sum_{d\mid r}c_{d}(\bar{n})\left(\prod_{i=1}^{\nu}\prod_{j=1}^{N(\varphi_{i})+1}\prod_{\alpha\in\hat\Delta_j^{i,-}(\varphi^0)}(1-e(d\alpha))^{-\frac{jr}{d}}\right)e((-n+b)\delta).
\end{align*}
\end{thmt}

\section{Example: $\mathfrak{sl}_3(\C)$ case}\label{sec:8}
Throughout this section, assume that $\mathfrak{g}=\mathfrak{sl}_3(\C)$, the complex simple Lie algebra of type $A_2$.
We apply our results to this case, and determine the characters of all irreducible Harish-Chandra $\widetilde{\mathfrak{g}}$-modules in Chari's category $\widetilde{\mathcal{O}}$. 

For $1\leq i,j\leq 3$, denote by $E_{i,j}$ the $3\times 3$ matrix whose only nonzero entry is the $(i,j)$-entry which is $1$.
Fix the Cartan subalgebra \[\mathfrak{h}=\C h_1\oplus \C h_2\] of $\mathfrak{g}$, where $h_i:=E_{i,i}-E_{i+1,i+1}$, $i=1,2$. 
Then 
\begin{align*}
    \Delta=\{\pm \alpha_1,~\pm \alpha_2,\pm (\alpha_1+\alpha_2)\}
\end{align*}
is the root system of $\fg$ with respect to $\fh$, where $\alpha_1,\alpha_2\in\fh^\ast$ satisfy
\begin{align*}
    \alpha_1(h_1)=\alpha_2(h_2)=2,\qquad \alpha_1(h_2)=\alpha_2(h_1)=-1.
\end{align*}
Note that  $h_1,~h_2$ are  the coroots of $\alpha_1,\alpha_2$, respectively. 
Take $\Pi=\{\alpha_1,\alpha_2\}$, and the corresponding triangular decomposition of $\fg$ is given by 
$\fg=\fn^-\oplus\fh\oplus \fn^+$ with 
\[\fn^+=\C E_{1,2}\oplus \C E_{1,3}\oplus \C E_{2,3}\quad \te{and}\quad \fn^-=\C E_{2,1}\oplus \C E_{3,1}\oplus \C E_{3,2}.\]
For convenience, we set \[\alpha_0:=\alpha_1+\alpha_2.\]

By Proposition \ref{fexp}, every irreducible Harish-Chandra $\widetilde{\mathfrak{g}}$-module in Chari's category $\widetilde{\mathcal{O}}$ is of  the form $\widetilde{V}(\varphi,b)$ for some $(\varphi,b)\in\mathcal{E}\times \mathbb{C}$. In view of Propositions \ref{pr:char-affine} and \ref{pr:pbpichi}, the computation of their characters reduces to that of irreducible highest weight modules over Takiff algebras associated with $\mathfrak{g}$. We therefore first determine
the character formula of the irreducible $\mathcal{C}_N\g$-module $V_{\CC_N\b}(\chi)$ when $\chi$ is not nonsingular, while the nonsingular case is obtained in Theorem \ref{th:chartak}.

Let $0\neq \chi\in(\CC_N\fh)^\ast$ be such that 
\begin{align}\label{chigeneiue}
\chi(h_1(N))\neq 0\quad \te{or}\quad \chi(h_2(N))\neq 0.
\end{align}
Recall from Subsection \ref{sec:chartak} that for $i=1,2$,
\[\chi_i:=\chi_{\alpha_i}=\te{max}\{j+1\mid 0\leq j\leq N,~\chi(h_i(j))\neq 0\}\]
with the convention that $\chi_{\alpha_i}=0$ if $\chi(h_i(j))=0$ for all $j=0,1,\dots,N$. Similarly, define
\[
\chi_0
:=
\max\{j+1\mid 0\leq j\leq N,\ 
\chi((h_1+h_2)(j))\neq0\},
\]
with the same convention. Then \eqref{chigeneiue} implies that $\te{max}\{\chi_1,\chi_2\}=N+1,$
and hence the associated $(N+2)$-partition $\bm\Pi(\chi)$ of $\Pi$ is genuine. We also have 
\[\fh_{\bm\Pi(\chi)}=\sum_{j=0}^{\chi_1-1}\C h_1(j)+\sum_{j=0}^{\chi_2-1}\C h_2(j).\]
For later use, we set 
\[c_{\chi,1}=\chi(h_1(0))+1,\quad c_{\chi,2}=\chi(h_2(0))+1,\quad \te{and}\quad c_{\chi,0}=\chi((h_1+h_2)(0))+2.\]
For $i=0,1,2$, set
\begin{align*}
    \theta_{\chi,i}:=\begin{cases}
        1,&\mbox{if }c_{\chi,i}\in\Z_{+},\\
        0,&\mbox{otherwise}.
    \end{cases}
\end{align*}

When $N=0$, we identify $\CC_N\fg$ with $\fg$. The following character formula holds for $V_{\CC_0\b}(\chi)$ whether $\chi$ is 
nonsingular  or not (cf.~\cite{H1}):
\begin{align}
 \textrm{char\,} V_{\CC_N\b}(\chi)
= &\frac{e(\chi)}{D(\alpha_1,\alpha_2)}\cdot \left(1-\theta_{\chi,1}e(-c_{\chi,1}\alpha_1)-\theta_{\chi,2}e(-c_{\chi,2}\alpha_2)-\theta_{\chi,0}e(-c_{\chi,0}\alpha_0)\right.\notag\\ 
&\left.+\theta_{\chi,1}\theta_{\chi,0}e(-c_{\chi,1}\alpha_1
-c_{\chi,0}\alpha_2)+\theta_{\chi,2}\theta_{\chi,0}e(-c_{\chi,2}\alpha_2
-c_{\chi,0}\alpha_1)\right).
\end{align}
Here, \[D(\alpha_1,\alpha_2):={(1-e(-\alpha_1))(1-e(-\alpha_2))(1-e(-\alpha_0))}.\]

Now assume that $N>0$. Set 
 \[m_\chi:=\chi_0+\delta_{\chi_0,0}-\delta_{\chi_0,1}\delta_{\chi((h_1+h_2)(0)),-N-1}.\]
 By the definition of nonsingularity,  $\chi|_{\fh_{\bm\Pi(\chi)}}$ is not nonsingular if and only if  one of the following conditions holds:
\begin{itemize}
    \item[(S1)] $\chi_1=1$ and $c_{\chi,1}\in\mathbb{Z}_{+}$;
    \item[(S2)] $\chi_2=1$ and $c_{\chi,2}\in\mathbb{Z}_{+}$;
    \item[(S3)] $\chi_1=\chi_2 = N + 1$ and $\chi_0\leq N$.
\end{itemize}
To derive a character formula in these cases, we need to characterize the maximal proper $\mathcal{C}_N\fg$-submodule of the parabolically
induced module
$
M_{\CP_{\bm\Pi(\chi)}}
\bigl(\chi|_{\fh_{\bm\Pi(\chi)}}\bigr)$.
To this end, we introduce the set
\[S_{\chi}:=\{E_{2,1}(0)^{c_{\chi,1}}v_\chi\}\quad(\te{resp}.\;S_{\chi}:=\{E_{3,2}(0)^{c_{\chi,2}}v_\chi\})\]
whenever $\chi$ satisfies (S1) (resp.\;(S2)), as well as   \begin{align*}
    S_\chi:=\left\{\left(\sum_{i=s}^{N}\chi(h_2(N+s-i))E_{3,1}(i)-\sum_{\substack{0\leq i,j\leq N,\\ i+j=N+s}}E_{2,1}(i)E_{3,2}(j)\right)v_\chi\mid m_\chi\leq s\leq N\right\}
\end{align*}
 whenever $\chi$ satisfies (S3). 
 One easily verifies that every element in $S_\chi$ is annihilated by $\CC_N\fn^+$.

Furthermore, we have (cf. \cite{HLN})
\begin{lemt}\label{irrchi}
The parabolically induced $\CC_N\fg$-module $M_{\mathcal{P}_{\bm\Pi(\chi)}}(\chi|_{\fh_{\bm\Pi(\chi)}})$ is reducible if and only if $\chi$ satisfies $\mathrm{(S1)}$, $\mathrm{(S2)}$, or $\mathrm{(S3)}$, in which case the maximal proper submodule is generated by the set $S_\chi$.
\end{lemt}

 Assume that $\chi|_{\fh_{\bm\Pi(\chi)}}$ is not nonsingular. By Lemma \ref{irrchi}, 
 \[V_{\CC_N\mathfrak{b}}(\chi)\cong M_{\mathcal{P}_{\bm\Pi(\chi)}}(\chi|_{\fh_{\bm\Pi(\chi)}})/ \langle S_\chi\rangle\]
 as $\CC_N\fg$-modules.
 Here,
 $\langle S_\chi \rangle$ denotes the  $\CC_N\mathfrak{g}$-submodule of $M_{\mathcal{P}_{\bm\Pi(\chi)}}(\chi|_{\fh_{\bm\Pi(\chi)}})$ generated by  $S_\chi$.
 We now use this realization to compute 
 its character formula.  
 
 Define a subspace of $\CC_N\fg$ as follows:
 \begin{align*}
     \fn_\chi:=\begin{cases}
         \sum_{k=0}^N\C E_{3,2}(k)+\sum_{k=0}^N\C E_{3,1}(k),&\mbox{if }\chi\mbox{ satisfies (S1),}\\
         \sum_{k=0}^N\C E_{2,1}(k)+\sum_{k=0}^N\C E_{3,1}(k),&\mbox{if }\chi\mbox{ satisfies (S2),}\\
         \sum_{k=0}^N\C E_{2,1}(k)+\sum_{k=0}^N\C E_{3,2}(k)+\sum_{k=0}^{m_\chi-1}\C E_{3,1}(k),&\mbox{if }\chi\mbox{ satisfies (S3).}
     \end{cases}
 \end{align*}
Since $\fn_\chi$ is invariant under the adjoint action of $\fh$, it induces  a natural $\fh$-module structure on its symmetric algebra $S(\fn_\chi)$.

 When $\chi$ satisfies (S1), by
 \eqref{eq:basispara},  a basis of $M_{\mathcal{P}_{\bm\Pi(\chi)}}(\chi|_{\fh_{\bm\Pi(\chi)}})$ is given by  
\begin{align}\label{base1}
 \left\{   \prod_{s=0}^N(E_{3,1}(s)^{k_{3,s}}E_{3,2}(s)^{k_{2,s}})E_{2,1}(0)^{k_{1,0}}v_\chi\mid k_{1,0},k_{2,s},k_{3,s}\in\mathbb{N},~0\leq s\leq N\right\}.
\end{align} 
It is straightforward to see that $\langle S_\chi \rangle$ is spanned by the elements in \eqref{base1} with $k_{1,0}\geq c_{\chi,1}$ in this case. Hence, there is an $\fh$-module isomorphism 
\begin{align}\label{sing1}
    V_{\CC_N\mathfrak{b}}(\chi) \cong \oplus_{k=0}^{c_{\chi,1}-1}\left(\C_{\chi|_{\fh}-k\alpha_1}\otimes S(\fn_\chi)\right).
\end{align} 
When $\chi$ satisfies (S2), an argument  similar to that  for (S1) gives 
\begin{align}\label{sing2}
    V_{\CC_N\mathfrak{b}}(\chi) \cong \oplus_{k=0}^{c_{\chi,2}-1}\left(\C_{\chi|_{\fh}-k\alpha_2}\otimes S(\fn_\chi)\right).
\end{align}
Here, for every $\mu\in\fh^\ast$, $\C_\mu:=\C v_\mu$ denotes the one-dimensional $\fh$-module such that
\begin{align*}
    h.v_\mu=\mu(h)v_\mu\quad \te{for }h\in\fh.
\end{align*}

When $\chi$ satisfies (S3), we have $\mathcal{P}_{\bm\Pi(\chi)}=\CC_N\mathfrak{b}$ 
and hence  $M_{\mathcal{P}_{\bm\Pi(\chi)}}(\chi|_{\fh_{\bm\Pi(\chi)}})$ is  the Verma $\CC_N\fg$-module of highest weight $\chi$. Note that 
\[\langle S_\chi \rangle=\sum_{s=m_\chi}^N \mathcal{U}(\CC_N\fn^-){\bf{e}}_sv_\chi,\]
where 
\[{\bf{e}}_s=E_{3,1}(s)-\sum_{\substack{0\leq i,j\leq N,\\ i+j\geq N+s}}c_{i,j}E_{2,1}(i)E_{3,2}(j)\]
for some $c_{i,j}\in\mathbb{C}$. It follows that
\begin{align}\label{sing3}
   M_{\mathcal{P}_{\bm\Pi(\chi)}}(\chi|_{\fh_{\bm\Pi(\chi)}})/\langle S_\chi\rangle &\cong \C_{\chi|_{\fh}}\otimes \mathcal{U}(\CC_N\fn^-)/\left(\sum_{s=m_\chi}^N \mathcal{U}(\CC_N\fn^-)\bf{e}_s\right)\notag\\
   &\cong \C_{\chi|_{\fh}}\otimes S(\fn_\chi)
\end{align}
as $\fh$-modules.

\begin{prpt}
Let $N\geq 1$ and let $\chi\in(\CC_N\fh)^\ast$ be such that \begin{align}
\chi(h_1(N))\neq 0\quad \te{or}\quad \chi(h_2(N))\neq 0.
\end{align}
 Then 
    \begin{align*}
    \mathrm{char\,}V_{\CC_N\mathfrak{b}}(\chi)=\frac{e(\chi|_{\fh})(1-\delta_{\chi_1,1}\theta_{\chi,1}e(-c_{\chi,1}\alpha_1)-\delta_{\chi_2,1}\theta_{\chi,2}e(-c_{\chi,2}\alpha_2))}{(1-e(-\alpha_1))^{\chi_1}(1-e(-\alpha_2))^{\chi_2}(1-e(-\alpha_0))^{m_\chi}}.
    \end{align*}
\end{prpt}
\begin{proof}
    When $\chi|_{\fh_{\bm\Pi(\chi)}}$ is nonsingular, Theorem \ref{th:chartak} gives that 
\begin{align*}
    \mathrm{char\,}V_{\CC_N\mathfrak{b}}(\chi)=e(\chi|_{\fh})(1-e(-\alpha_1))^{-\chi_1}(1-e(-\alpha_2))^{-\chi_2}(1-e(-\alpha_0))^{-N-1}.
\end{align*}
Then the assertion follows immediately since  $m_\chi=N+1$ in this case. When $\chi|_{\fh_{\bm\Pi(\chi)}}$ is not nonsingular, the assertion follows from \eqref{sing1}-\eqref{sing3}.
\end{proof}

We now apply the preceding character formulas to the irreducible Harish-Chandra $\widetilde{\mathfrak g}$-modules in Chari's category $\widetilde{\mathcal{O}}$. Let $b\in\C$ and $\varphi\in\mathcal E$.
If $\varphi=0$, then $\widetilde V(0,b)$ is one-dimensional and
\[
\operatorname{char}\widetilde V(0,b)=e(b\delta).
\]
Henceforth, assume that $0\neq\varphi\in\mathcal E$ and
$A_\varphi=r\Z$ for some $r\in\Z_+$ with decomposition \[\varphi=\varphi^0+\varphi^1+\dots+\varphi^{r-1}\] as in \eqref{eq:decomvarphi}. Here, $\varphi^0$ is an exp-polynomial linear functional on $\CL\fh$ such that there exist distinct nonzero complex numbers $a_1,a_2,\dots,a_{\nu}$ with $(a_p/a_q)^r\neq 1$ for $1\leq p\neq q\leq \nu$, and $f_{\alpha_1,j}(t),f_{\alpha_2,j}(t)\in\C[t]$ for $1\leq j\leq \nu$ satisfying 
\begin{align*}
    \varphi^0(h_i\otimes t^n)=\sum_{j=1}^\nu a_j^nf_{\alpha_i,j}(n)\quad \te{for }i=1,2,~n\in\Z.
\end{align*}
As in Subsection \ref{sec:charloop}, we have 
\[N_j=\te{max} \{\deg f_{\alpha_1,j}(t),~\deg f_{\alpha_2,j}(t)\}\quad \]
 and  characters $\bar{\varphi}_{j}$, $1\leq j\leq \nu$, on $\CC_{N_j}\fh$
such that 
\begin{align*}
\bar{\varphi}_{j}(h_i(k))=\sum_{k'=0}^k(-1)^{k-k'}\binom{k}{k'}f_{\alpha_i,j}(k')\quad \te{for } i=1,2,~0\leq k\leq N_j.
\end{align*}
For every $j=1,2,\dots,\nu$, the character $\bar{\varphi}_{j}$ satisfies 
\[\bar{\varphi}_{j}(h_1(N_j))\neq 0\quad \te{or}\quad \bar{\varphi}_{j}(h_2(N_j))\neq 0.\]

Set 
\[Z:=\{1\leq j\leq \nu\mid N_j=0\},\quad \hat{Z}:=\{1\leq j\leq \nu\mid N_j\neq 0\},\]
and write $|Z|$ for the cardinality of $Z$.
By \eqref{char:loop-tak}, we have 
\begin{align}\label{charsl3}
    &\mathrm{char}\, V(\varphi^0)\notag\\
    =\ &  \left(\prod_{j\in Z} \mathrm{char}\, V_{\CC_{N_j}\mathfrak{b}}(\bar\varphi_j)\right)\left(\prod_{j\in\hat{Z}}\te{char}\, V_{\CC_{N_j}\mathfrak{b}}(\bar\varphi_j)\right)\notag\\
    =\ &e(\varphi^0|_{\fh})(1-e(-\alpha_0))^{-|Z|-\sum_{j\in\hat{Z}}m_{{\bar\varphi}_j}} \left(\prod_{i=1}^2((1-e(-\alpha_i))^{-|Z|-\sum_{j\in\hat{Z}}({\bar\varphi}_j)_{\alpha_i}})\right)\notag\\
&\cdot \prod_{j\in Z}\left(1+\theta_{{\bar\varphi}_j,0}\sum_{i=1}^2\theta_{{\bar\varphi}_j,i}e(-c_{\bar\varphi_j,0}\alpha_0+(c_{\bar\varphi_j,0}-c_{\bar\varphi_j,i})\alpha_i)-\sum_{i=0}^2\theta_{{\bar\varphi}_j,i}e(-c_{{\bar\varphi}_j,i}\alpha_i)\right)\notag\\
&\cdot \left(\prod_{j\in \hat{Z}}(1-\sum_{i=1}^2\delta_{({\bar\varphi}_j)_{\alpha_i},1}\theta_{{\bar\varphi}_j,i}e(-c_{{\bar\varphi}_j,i}\alpha_i)\right).
\end{align}

Combining \eqref{charsl3} with Proposition \ref{pr:char-affine},
we obtain the following character formula.
\begin{thmt}
 Let $b\in\C$ and  let $ 0 \neq \varphi\in\mathcal E$ satisfy
$A_\varphi=r\Z$ for some $r\in\Z_+$. Assume that $\varphi$ has the decomposition \eqref{eq:decomvarphi}.  Then 
\begin{align*}
\mathrm{char}\,\wt{V}(\varphi,b)\ =\ \frac{1}{r}\sum_{n\in\Z}\sum_{d\mid r}c_{d}(\bar{n})(\mathrm{char}\,V(\varphi^0,d))^{\frac{r}{d}}e((-n+b)\delta),
\end{align*}
where $\mathrm{char}\,V(\varphi^0,d)$ is obtained from $\mathrm{char}\,V(\varphi^0)$ by replacing $e(\alpha)$ by $e(d\alpha)$ for every $\alpha\in\fh^\ast$ and $\mathrm{char}\,V(\varphi^0)$ is given by \eqref{charsl3}. 
\end{thmt}
\section*{Acknowledgment}
Fulin Chen was supported by the Natural Science Foundation of Xiamen, China
(Grant No. 3502Z202473005), the National Natural Science Foundation of China (Grant
No. 12471029), and the Natural Science Foundation of Fujian Province, China (Grant No. 2026J001042). Xin Huang was supported by the National Natural Science Foundation of China (Grant No. 12526548).

ChatGPT (GPT-6 Astra) was used as an editorial aid to refine the language of this manuscript. Its use was limited to language refinement.

\end{document}